\documentclass[final,1p,times,sort&compress]{elsarticle}%
\usepackage{amsmath}
\usepackage{amsfonts}
\usepackage{amssymb}
\usepackage{hyperref}

\usepackage{color}%
\usepackage{subfigure}

\newtheorem{mytheo}{Theorem}[section]

\newtheorem{cor}[mytheo]{Corollary}
\newtheorem{lem}[mytheo]{Lemma}

\newtheorem{algo}[mytheo]{Algorithm}

\newtheorem{rem}[mytheo]{Remark}

\newcommand{\abs}[1]{\left\vert#1\right\vert}

\newcommand{\eps}{\varepsilon}

\newcounter{remark}

\newcounter{problem}
\newenvironment{problem}{\refstepcounter{problem}\vspace{1.5ex}
	{\noindent\bf Problem
		\theproblem.}\hspace{0.3em}\parindent=0pt}{\vspace{1ex}}

\makeatletter
\def\@upcite#1#2{\textsuperscript{[{#1\if@tempswa , #2\fi}]}}
\makeatother
\newenvironment{proof}{\vspace{1ex}
	{\it Proof. }\hspace{0.3em}}{\vspace{1ex}} \journal{ }
\begin{document}
	\begin{frontmatter}
		\title{Long term analysis of splitting methods
			for charged-particle dynamics}

		\author[1]{Xicui Li}
		\author[1]{Bin Wang\corref{cor1}}

		\address[1]{ School of Mathematics and Statistics, Xi'an Jiaotong University, 710049 Xi'an, China}

		\ead{lixicui@stu.xjtu.edu.cn,wangbinmaths@xjtu.edu.cn}\cortext[cor1]{Corresponding author.}

		
		\begin{abstract}
In this paper, we rigorously analyze the energy, momentum and magnetic moment behaviours of two splitting methods for solving charged-particle dynamics.  The near-conservations of these invariants   are given  {for the system under constant magnetic field or quadratic electric
potential}.  {By the approach named as backward error analysis, we derive} the modified  equations and modified invariants {of the splitting methods and based on which, the near-conservations over long times are proved.} Some numerical experiments are presented to demonstrate these long time behaviours.

		\end{abstract}
		
	\end{frontmatter}
 \vskip0.5cm \pagestyle{myheadings}
\thispagestyle{plain}
	\section{Introduction}
	In this paper, we {formulate and analyze} two splitting methods for solving the charged-particle dynamics (CPD)
{which is described by}
	\begin{equation}\label{CPD}
		\begin{split}
			&\ddot{x}=v\times B(x)+E(x),\quad \ x(0)=x_0, \quad v(0)=v_0,\quad t\ge 0,
		\end{split}
	\end{equation}
{with the  position $x(t)\in\mathbb{R}^3$ and the velocity $v(t):=\dot{x}(t)\in\mathbb{R}^3$,}	where $B(x)=\nabla_x\times A(x)\in \mathbb{R}^3$ {denoted by} $(B_1(x),B_2(x),B_3(x))^{\intercal}$ is {a non-uniform} {magnetic field which is determined by the vector potential $A(x)$}, and $E(x) =-\nabla_x U(x)$ is {a given electric field generated by some scalar potential $U(x)$.}
	Charged particle dynamics plays a fundamental role in plasma physics {(\cite{Birdsall}) and there has been a lot of effective numerical methods in recent years}
	  for solving this system. Boris algorithm \cite{70Relativistic} has been
	widely used in {the simulation} of magnetized {plasma} due to its excellent properties \cite{13Why,17A}. Many other kinds of methods have also been {developed and analyzed}, {such as  splitting methods \cite{02Splitting,Ostermann15}, exponential integrators
  \cite{EI1},  asymptotically preserving methods \cite{AP1,AP2}, filtered Boris algorithms \cite{lubich20},  uniformly accurate methods \cite{UA1,UA2,UA3} and so on.}

{On the other hand, in the past over three decades,  structure-preserving
algorithms  have arisen in {various fields} of applied
sciences. From the scientific computing point of view, it is
conventional to require numerical methods  to preserve the
qualitative features of the true solution as much as possible when
applied to a differential equation.  For the system of CPD \eqref{CPD}, we pay attention to  three} important invariants. With the Euclidean norm $\abs{ \cdot}$, {the total energy is denoted} by {(\cite{Brugnano20,19Energy})}
	\begin{equation}\label{energy def}
		H(x,v)=\frac{1}{2}\abs{v}^2+U(x).
	\end{equation}
	{It is well known that the solution of \eqref{CPD} exactly  conserves this energy}.
	{For the scalar and vector potentials, if they are assumed to satisfy  the invariance properties}
	\begin{equation}\label{in-prop}
		U(e^{\tau S}x)=U(x),\quad e^{-\tau S}A(e^{\tau S}x)=A(x), \quad \forall \tau \in \mathbb{R},
	\end{equation}
 the momentum
	\begin{equation*}
		M(x,v)=(v+A(x))^\intercal Sx
	\end{equation*}
	{is preserved by the solution of \eqref{CPD}} (\cite{18Energy}), where $S$ is a skew-symmetric matrix. Moreover, {if we consider the system \eqref{CPD} under the constant magnetic field, i.e. $B(x)\equiv B$} and $S$ satisfies $Sv=v\times B$, then {$A(x)=-\frac{1}{2}x\times B$ and the momentum becomes}
		{$ M(x,v)=v^\intercal (x\times B)-\frac{1}{2}\abs{x\times B}^2 .$}
%
The third invariant is {the following magnetic moment}
	\begin{equation*}
		I(x,v)=\frac{\abs{v\times B(x)}^2}{2\abs{B(x)}^3}=\frac{\abs{v_\perp}^2}{2\abs{B(x)}},
	\end{equation*}
{where $v_\perp:=\frac{v\times B(x)}{\abs{B(x)}}$ is orthogonal to $B(x)$.
Based on the analysis of \cite{1994Adiabatic,2009Hamiltonian,1963The Adiabatic}, it is  known} that {this} magnetic moment is an \emph{adiabatic invariant}.

{From the point of long time scientific computation,  it} is of great interest to investigate the long time behaviour of numerical methods. In order to get	 numerical methods with near/exact conservation  of
qualitative features of the CPD, various numerical methods have been formulated and analyzed.
Broadly speaking, up to
the present, four categories of structure-preserving algorithms for
the CPD have been in the center of research: energy-preserving methods \cite{Brugnano20,Li-AMC,Li-AML,Chacon20,22Energy,21Error,Wang21} to preserve the energy,
 volume-preserving {schemes} \cite{15Volume}  to preserve the volume, symplectic algorithms \cite{14Symplectic,16Explicit,16Zhang,17Explicit} to preserve the symplecticity and numerical methods  \cite{17Symmetric,Wang20,18Energy,20Long,00Long,16Long} with near conservation of qualitative features.
 Recently splitting methods have been considered for solving CPD  \cite{Ostermann15,21Error,22Energy} and energy-preserving splitting methods have been constructed.
  However, these publications  focus on the   analysis of the accuracy and energy-preserving property of splitting methods. It seems that
the long-time  analysis of  classical splitting methods has not been considered and
the behaviour of energy-preserving splitting methods concerning other
structure-preserving aspects has {not} been investigated in the
literature, such as the long-time numerical conservation of
momentum  and	magnetic moment. \emph{From the perspective of structure preserving algorithm, this is far from enough}.
  Motivated by this point, {this paper applies} backward error analysis for two splitting methods to gain insight into their long-term energy, momentum and magnetic moment performance.

%
	
    The rest of the paper is organized as follows. In section \ref{method}, we discribe two different symmetric splitting methods and study their local errors. Then,  in section \ref{mr&ne}, the  main results on the numeircal long time conservation are presented and three numerical experiments are carried out  to support the theoretical results. The complete analysis of the results are rigorously given in section \ref{proof}. The lastsection  is devoted to the conclusions of this paper.
	\section{Numerical Methods}\label{method}
	{In this section, we present {two splitting} methods and show their elementary features.}
	    {To formulate the methods, we   split the equation \eqref{CPD}} into two subflows:
	\begin{equation}\label{subflow}
		\frac{d}{dt }\begin{pmatrix}
			x \\
			v \\
		\end{pmatrix}= \begin{pmatrix}
			0 \\
			v\times B(x) \\
		\end{pmatrix}
		,\qquad  \frac{d}{dt}\begin{pmatrix}
			x \\
			v \\
		\end{pmatrix}
		= \begin{pmatrix}
			v \\
			E(x) \\
		\end{pmatrix}.
	\end{equation}
 {Then for the first subflow}, which is integrable, one can obtain its exact solution:
	\begin{equation}\label{LM}\Phi^{L}_{t}:\\ \left(
		\begin{array}{c}
			\xi(t) \\
			\eta(t) \\
		\end{array}
		\right)=\left(
		\begin{array}{c}
			\xi(0) \\
			e^{t\tilde{B} (\xi(0))}\eta(0) \\
		\end{array}
		\right),\quad t\geq0,
	\end{equation}
	where $\tilde{B}(x)=\begin{pmatrix}
		0 &B_3(x)& -B_2(x)\\
		-B_3(x) &0 &B_1(x)\\
		B_2(x) &-B_1(x) &0 \\
	\end{pmatrix}$
	is defined by $v\times B(x)=\tilde{B}(x)v$.
	For the second subflow (a canonical Hamiltonian system), we consider applying the average vector field (AVF) formula \cite{1999Geometric} to get its numerical propagator:
	\begin{equation}\label{ILM}
		\Phi^{iL}_{t}:\ \
		\begin{pmatrix}
			{\xi(t)} \\
			{\eta(t)} \\
		\end{pmatrix}
		=\begin{pmatrix}
			\xi(0)+t\eta(0)+\frac{t^2}{2}\int_{0}^1
			E\left( \rho \xi(0)+ (1-\rho)\xi(t) \right) d\rho \\
			\eta(0)+t\int_{0}^1
			E\left(\rho \xi(0)+ (1-\rho)\xi(t)\right) d\rho  \\
		\end{pmatrix}.
	\end{equation}

 {Based on this splitting, two methods are formulated as follows.}
	\begin{algo}[Implicit splitting method]\label{IMS}
		We denote the numerical solution as $x^n\approx x(t_n),\, v^n\approx v(t_n)$ and choose $x^0=x_0,\,v^0=v_0$. Taking a symmetric version \cite{2006Geometric}: $\Phi_h^i = \Phi_\frac{h}{2}^L \circ \Phi_h^{iL} \circ \Phi_\frac{h}{2}^L,$ {then we get its total formula} for solving (\ref{CPD}): for $n\geq0$,
		\begin{equation}\label{eq-IMS}
			\begin{aligned}x^{n+1}=&x^n+he^{\frac{h}{2}\tilde{B} (x^n)}v^n+\frac{h^2}{2}\int_{0}^{1}E\left(\rho x^n+(1-\rho)x^{n+1}\right)d\rho,\\
				v^{n+1}=&e^{\frac{h}{2}\left({\tilde{B} (x^{n+1})+\tilde{B} (x^n)}\right)}v^n+he^{\frac{h}{2}{\tilde{B} (x^{n+1})}}\int_{0}^{1}E\left(\rho x^n+(1-\rho)x^{n+1}\right)d\rho. \\
			\end{aligned}
		\end{equation}
		Obviously, it is an implicit and typical Strang splitting scheme, and we shall { use IMS-O2 to denote the method}.
	\end{algo}	
	
	{If we make some adjustments to \eqref{ILM}, we {can get the following} explicit algorithm.}
	\begin{algo}[Explicit splitting method]\label{EXS}
		For the second subflow in \eqref{subflow}, we linearize the nonlinear integrals in \eqref{ILM}, and then it is transformed to
		\begin{equation}\label{ELM}
			\Phi^{eL}_{t}:\\
			\begin{pmatrix}
				{\xi(t)} \\
				{\eta(t)} \\
			\end{pmatrix}
			=\begin{pmatrix}
				\xi(0)+t\eta(0)+\frac{t^2}{2}E(\xi(0)) \\
				\eta(0)+\frac{t}{2}\left[E(\xi(0))+E(\xi(t))\right]\\
			\end{pmatrix}.
		\end{equation}
		In that way, the Strang splitting scheme $\Phi_h^e:= \Phi_\frac{h}{2}^L \circ \Phi_h^{eL} \circ \Phi_\frac{h}{2}^L$  {for solving \eqref{CPD}} reads:
		\begin{equation}\label{eq-EXS}
			\begin{aligned}x^{n+1}=&x^n+he^{\frac{h}{2}\tilde{B} (x^n)}v^n+\frac{h^2}{2}E(x^n),\\
				v^{n+1}=&e^{\frac{h}{2}\left({\tilde{B} (x^{n+1})+\tilde{B} (x^n)}\right)}v^n+\frac{h}{2}e^{\frac{h}{2}{\tilde{B} (x^{n+1})}}\left[E(x^n)+E(x^{n+1})\right], \\
			\end{aligned}
		\end{equation}
		and we shall call it as EXS-O2.	
	\end{algo}

 {In what follows, we will study the basic properties of these two methods. We begin with their symmetry.}
A numerical method denoted by $y^{n +1} = \Phi_h (y^n)$ is called symmetric if exchanging $y^n\leftrightarrow y^{n +1}$ and $h\leftrightarrow -h$ leaves the method unaltered. {It has been pointed out in \cite{2006Geometric} that symmetric methods have excellent {long time} behavior and play a central role in the geometric numerical integration of differential equations. The following theorem states the symmetry of Algorithms \ref{IMS} and \ref{EXS}.}
\begin{mytheo}[Symmetry]
	 {From the symmetric versions: $\Phi_h^i = \Phi_\frac{h}{2}^L \circ \Phi_h^{iL} \circ \Phi_\frac{h}{2}^L$ and $\Phi_h^e = \Phi_\frac{h}{2}^L \circ \Phi_h^{eL} \circ \Phi_\frac{h}{2}^L$, it follows that Algorithms} \ref{IMS} and \ref{EXS} are symmetric.
\end{mytheo}
	For these {two algorithms, we {next investigate and briefly show their local errors.} }
		 \begin{mytheo}[Local errors]\label{Local errors} For the two methods presented above, {under  the local assumptions $x(t_n)=x^n, v(t_n)=v^n,$  the local errors are given as
	 \begin{equation*}
			x^{n+1}-x(t_{n+1})=\mathcal{O}(h^3),\quad v^{n+1}-v(t_{n+1})=\mathcal{O}(h^3),
		\end{equation*}  {where the constants} symbolized by $\mathcal{O}$ are independent of $n$ {and} $h$.}
	\end{mytheo}
	\begin{proof}
{For simplicity of notations, we denote
		\begin{equation*}
		y:=\begin{pmatrix}
		 		x\\v
		 	\end{pmatrix},\quad	 f(y):=\begin{pmatrix}
				v\\v\times B(x)+E(x)
			\end{pmatrix},\quad f_1(y):=\begin{pmatrix}
				0\\v\times B(x)
			\end{pmatrix},\quad f_2(y)=\begin{pmatrix}
				v\\E(x)
			\end{pmatrix},
		\end{equation*}
 and let} $\Phi_h$  stand for the exact flow. Then the original system can be rewritten as  $\dot{y}=f(y)=f_1(y)+f_2(y)$.
	
		By Taylor expansions, we  reach
	\begin{align*}
		\Phi_h(y)=\left[Id+hf+\frac{h^2}{2}f'f\right](y)+\mathcal{O}(h^3),\quad \Phi_h^L(y)=\left[Id+hf_1+\frac{h^2}{2}f'_1f_1\right](y)+\mathcal{O}(h^3),\\
		\Phi_h^{iL}(y)=\left[Id+hf_2+\frac{h^2}{2}f'_2f_2\right](y)+\mathcal{O}(h^3),\quad \Phi_h^{eL}(y)=\left[Id+hf_2+\frac{h^2}{2}f'_2f_2\right](y)+\mathcal{O}(h^3).
	\end{align*}
Note that $\Phi_h^{iL}(y)\neq\Phi_h^{eL}(y)$ due to their different {coefficients} of the term $\mathcal{O}(h^3)$. After some calculations, we get
	{\begin{align*}
		\Phi_h^i(y)&=\Phi_\frac{h}{2}^L \circ \Phi_h^{iL} \circ \Phi_\frac{h}{2}^L(y) =\left[Id+h(f_1+f_2)+\frac{h^2}{2}(f'_1f_1+f'_1f_2+f'_2f_1+f'_2f_2)\right](y
		)+\mathcal{O}(h^3)\\
		&=\Phi_h(y)+\mathcal{O}(h^3).
	\end{align*}}Even if EXS-O2 is the linearization of IMS-O2, the expansions of $\Phi_h^{iL}$ and $\Phi_h^{eL}$ are similar. Hence, it can be {verified} that this makes no difference to the error result, {i.e.,} {$
	\Phi_h^e(y)-\Phi_h(y)=\mathcal{O}(h^3).$}
	{\hfill $ \blacksquare$}\end{proof}
	\section{Main results and numerical experiments}\label{mr&ne}
	\subsection{Main results}\label{mr}
	{In this subsection},  {we assume that the scalar potential $U(x)$ and  the magnetic field $B(x)$, as functions of $x$, are arbitrarily differentiable. In this paper,  $B(x)\equiv B$ represents the constant magnetic field and $U(x)=\frac{1}{2}x^\intercal Qx + q^\intercal x$ with a symmetric matrix $Q$ refers to the quadratic scalar potential.}  {Meanwhile, it is assumed that there exists a compact set  (independent of
	 $h$)  such that the result $(x^n, v^n)$ produced by the considered method stays in this set.} Then the energy, momentum and magnetic moment behaviours of Algorithms  \ref{IMS} and \ref{EXS} are given as follows one by one.
	\begin{mytheo}[Energy conservation]\label{energy}The method IMS-O2 exactly {preserves the energy \eqref{energy def}} of CPD \cite{22Energy}, i.e.
		\begin{equation*}
			IMS-O2:\qquad H(x^n,v^n)=H(x^0,v^0) \qquad for \qquad nh\le T.
		\end{equation*}
	   {We assume that  $B(x)\equiv B$ or $U(x)=\frac{1}{2}x^\intercal Qx + q^\intercal x$, then} {the energy \eqref{energy def}}  along the numerical solution over long {times} is conserved as follows
		\begin{equation}\label{eq-energy}
		EXS-O2:\ \ 	{H(x^n,v^n)=H(x^0,v^0)+\mathcal{O}(h^2)}\quad\text{for}\ \  nh\le h^{2-N},
		\end{equation}
{ with an arbitrarily large positive integer $N$.}

{Moreover, the method EXS-O2 has an exact modified energy conservation {which is stated as below.}
{If the scalar potential $U(x)$ is quadratic, EXS-O2 \eqref{eq-EXS} exactly preserves the modified energy
		\begin{equation}\label{Hh}
			H_h(x,v)=\frac{1}{2}\abs{v}^2+U(x)-\frac{h^2}{8}\abs{\nabla U(x)}^2
		\end{equation}
		at the discrete level, i.e.
		\begin{equation}\label{Hherr}	EXS-O2:\qquad H_h(x^{n+1},v^{n+1})=H_h(x^n,v^n)\qquad for \qquad nh\le T.
		\end{equation} }}
	\end{mytheo}
	\begin{mytheo}[Momentum conservation]\label{momentum}
	 { Under the conditions \eqref{in-prop} and  the following two assumptions:}
		\begin{itemize}
			\item  {$B(x)\equiv B$} and $Sv=v\times B$,
			\item {$U(x)=\frac{1}{2}x^\intercal Qx + q^\intercal x$ } and $QS=SQ$,
		\end{itemize}
 the momentum $M(x,v)=(v+A(x))^\intercal Sx$ along the numerical solution  {over long times is nearly preserved as}
		\begin{equation}\label{eq-momentum}
			{M(x^n,v^n)=M(x^0,v^0)+\mathcal{O}(h^2)} \quad \ \text{for}\quad \  nh\le h^{2-N},
		\end{equation}	
{where $N$ is an arbitrarily large positive integer.}
	\end{mytheo}
	\begin{mytheo}[Magnetic moment conservation] \label{mag-mom}
		 Assume that {the following two assumptions hold:}
		\begin{itemize}
			\item {$B(x)\equiv B$ },
			\item {$U(x)=\frac{1}{2}x^\intercal Qx + q^\intercal x$} and $Q\widehat{B}=\widehat{B}Q$ with $v\times \frac{B}{\abs{B}}=\widehat{B}v$,
		\end{itemize}
		then {the near conservation of magnetic moment $I(x,v)$ is}
		\begin{equation}\label{eq-mag-mom}
			{I(x^n,v^n)=I(x^0,v^0)+\mathcal{O}(h^2)} \quad\text{for}\quad nh\le h^{2-N},
		\end{equation}
{with an arbitrarily large positive integer $N$.}
	\end{mytheo}

{\subsection{Numerical experiments}}
	In this part, by means of MATLAB, we present {three numerical experiments} to show the  long time  energy, momentum, and magnetic moment behaviour of the above splitting  methods. We choose Boris method for comparison, and {solve all the tests on $[0,10000]$ with the step size $h=0.01$ to show the long time conservations.} For IMS-O2, the fixed-point iteration is employed into implicit iteration, where the error tolerance is $10^{-16}$ and  the maximum number of each iteration is 50. In addition, we make use of Gauss-Legendre to deal with the nonlinear integral in \eqref{eq-IMS}. {Denote the relative errors of energy, momentum and magnetic moment respectively by}
	\begin{equation}\label{error} e_H:=\frac{\abs{H(x^n,v^n)-H(x^0,v^0)}}{\abs{H(x^0,v^0)}},\ \ e_M:
=\frac{\abs{M(x^n,v^n)-M(x^0,v^0)}}{\abs{M(x^0,v^0)}},\ \ e_I:=\frac{\abs{I(x^n,v^n)-I(x^0,v^0)}}{\abs{I(x^0,v^0)}}.
	\end{equation}
  {For the momentum $M(x,v)=(v+A(x))^\intercal Sx$, { throughout} this {subsection}, the skew-symmetric matrix is given as} $S=\begin{pmatrix}
			0 &1& 0\\
			-1 &0 &0\\
			0 &0 &0 \\
		\end{pmatrix}$
	, i.e. $M(x,v)=(v_1+A_1(x))x_2-(v_2+A_2(x))x_1.$
	
\begin{figure}[t!]
		\centering
		\begin{tabular}[c]{ccc}%
			\subfigure{\includegraphics[width=4.7cm,height=4.2cm]{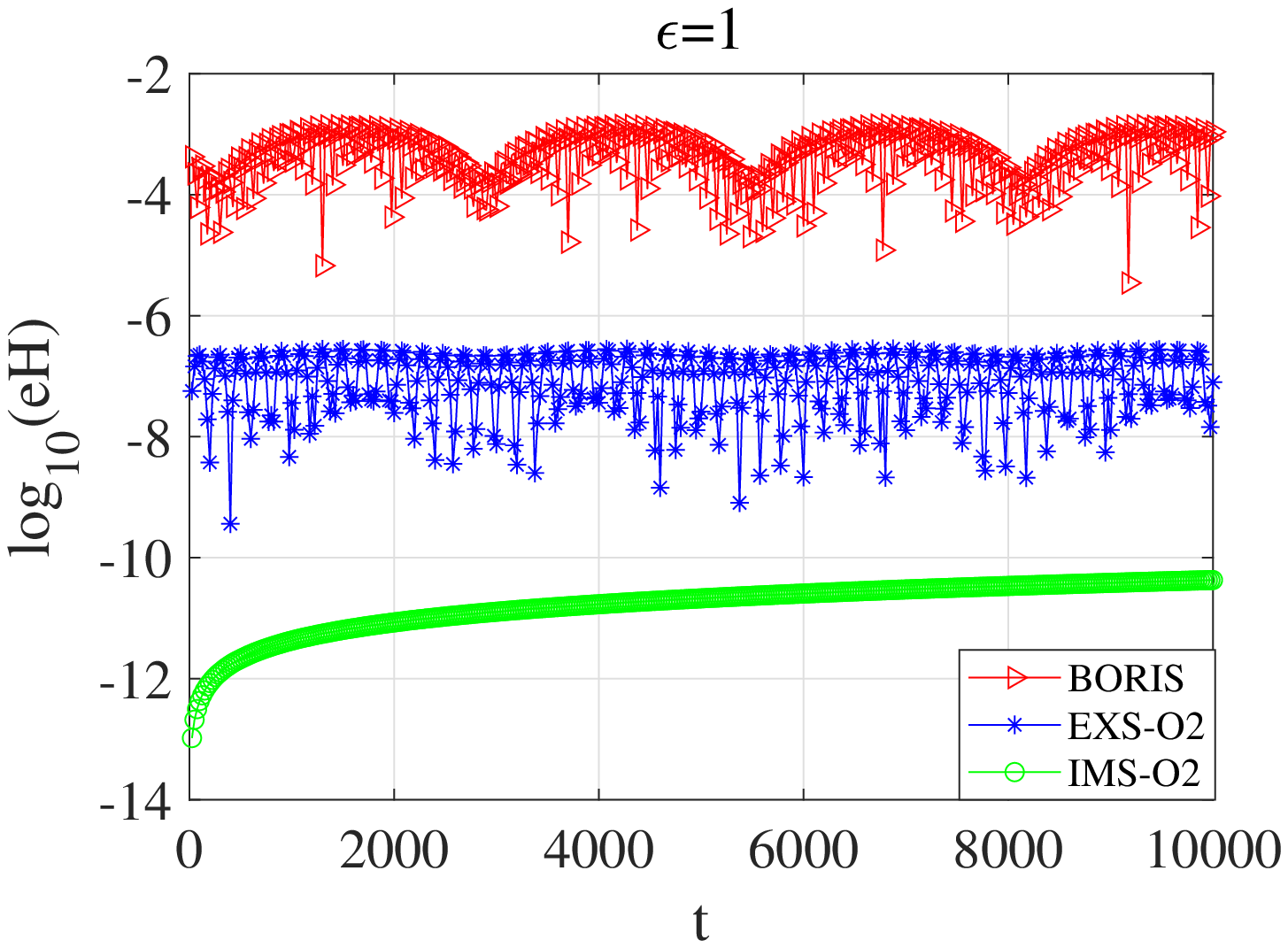}}			\subfigure{\includegraphics[width=4.7cm,height=4.2cm]{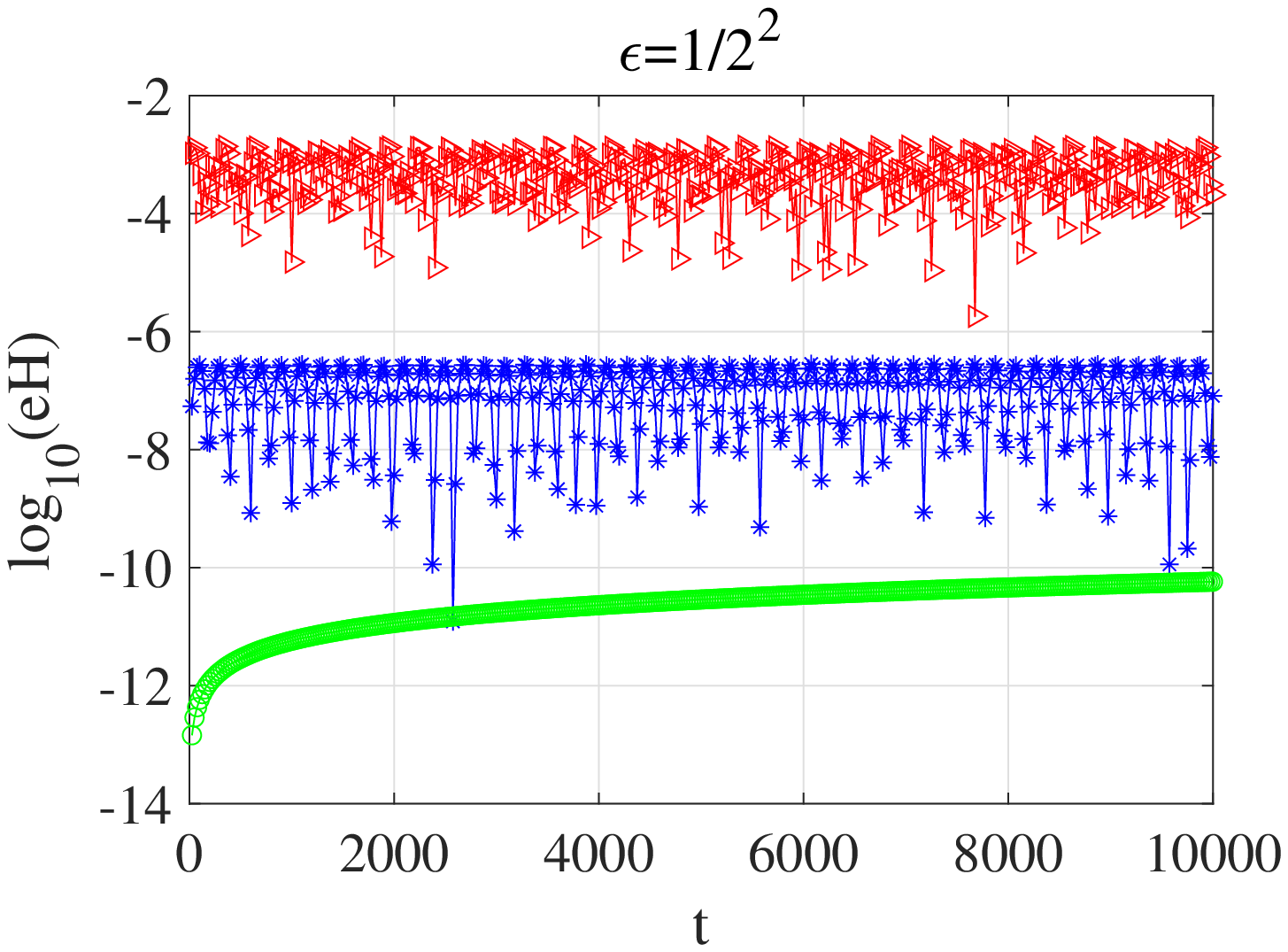}}
			\subfigure{\includegraphics[width=4.7cm,height=4.2cm]{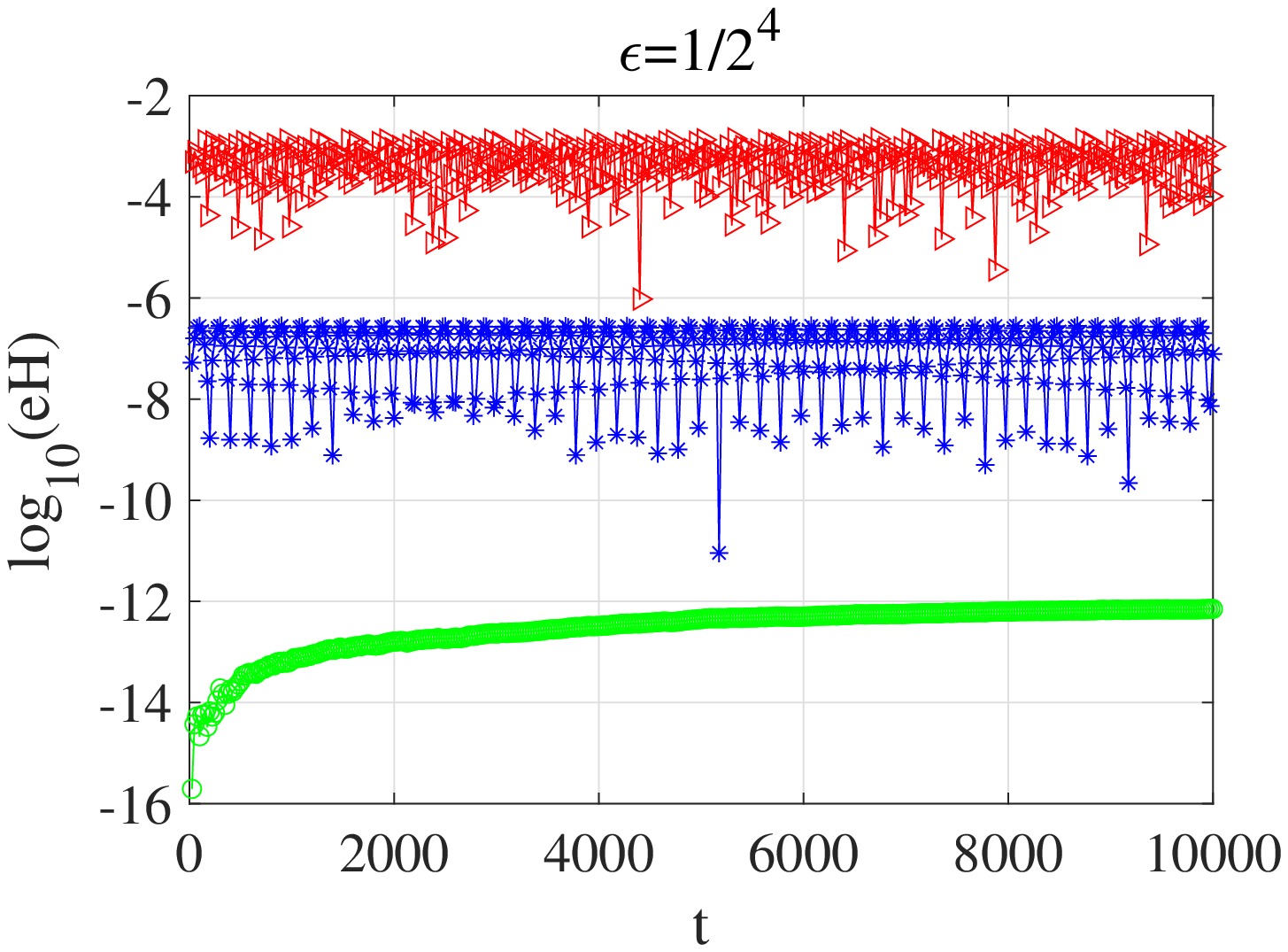}}		
		\end{tabular}
		\caption{{Problem 1.} Evolution of the energy error $e_H$ as function of time $t$. }\label{fig-eH3}
	\end{figure}

\begin{figure}[t!]
		\centering
		\begin{tabular}[c]{ccc}%
			\subfigure{\includegraphics[width=4.7cm,height=4.2cm]{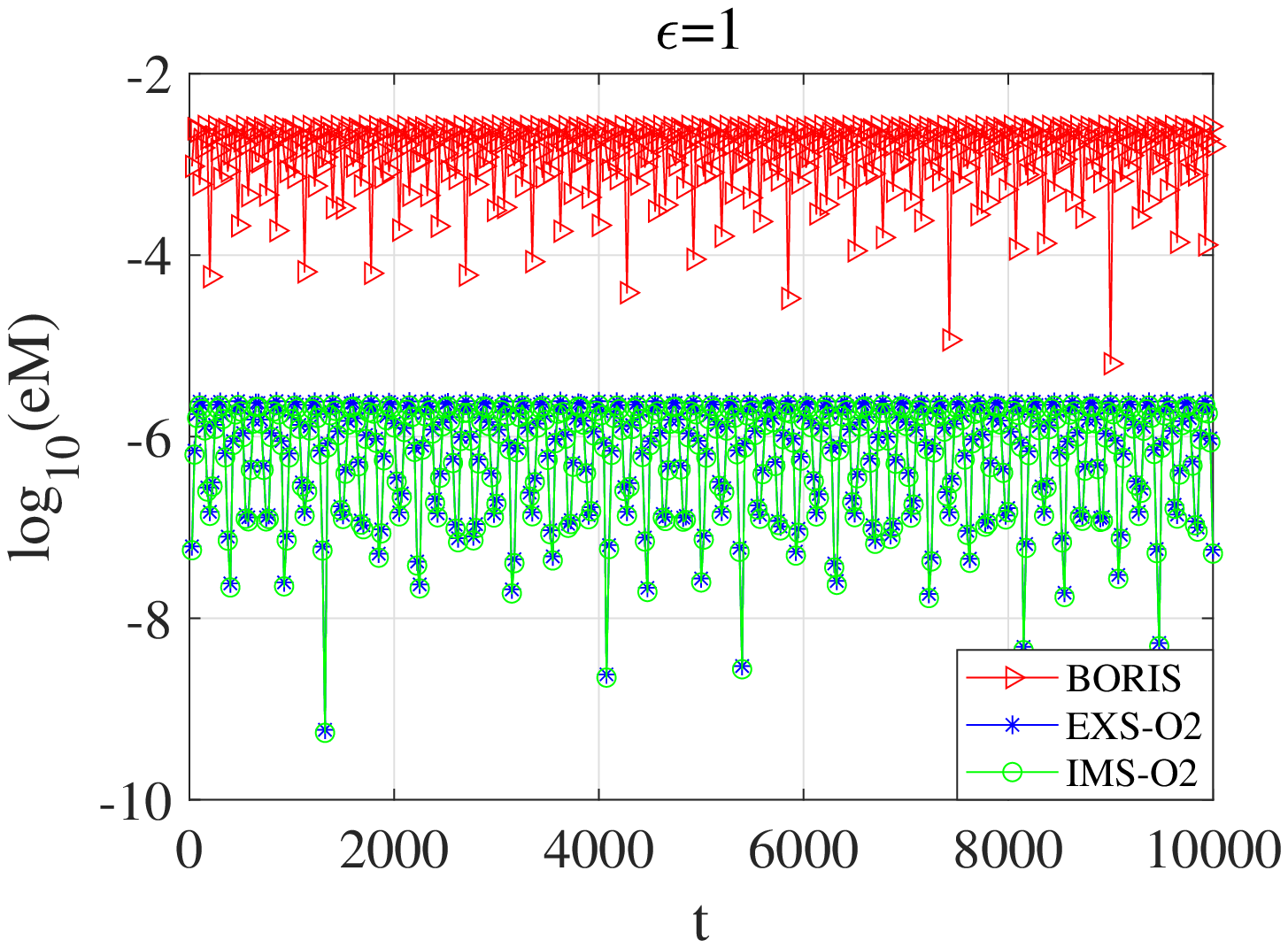}}			\subfigure{\includegraphics[width=4.7cm,height=4.2cm]{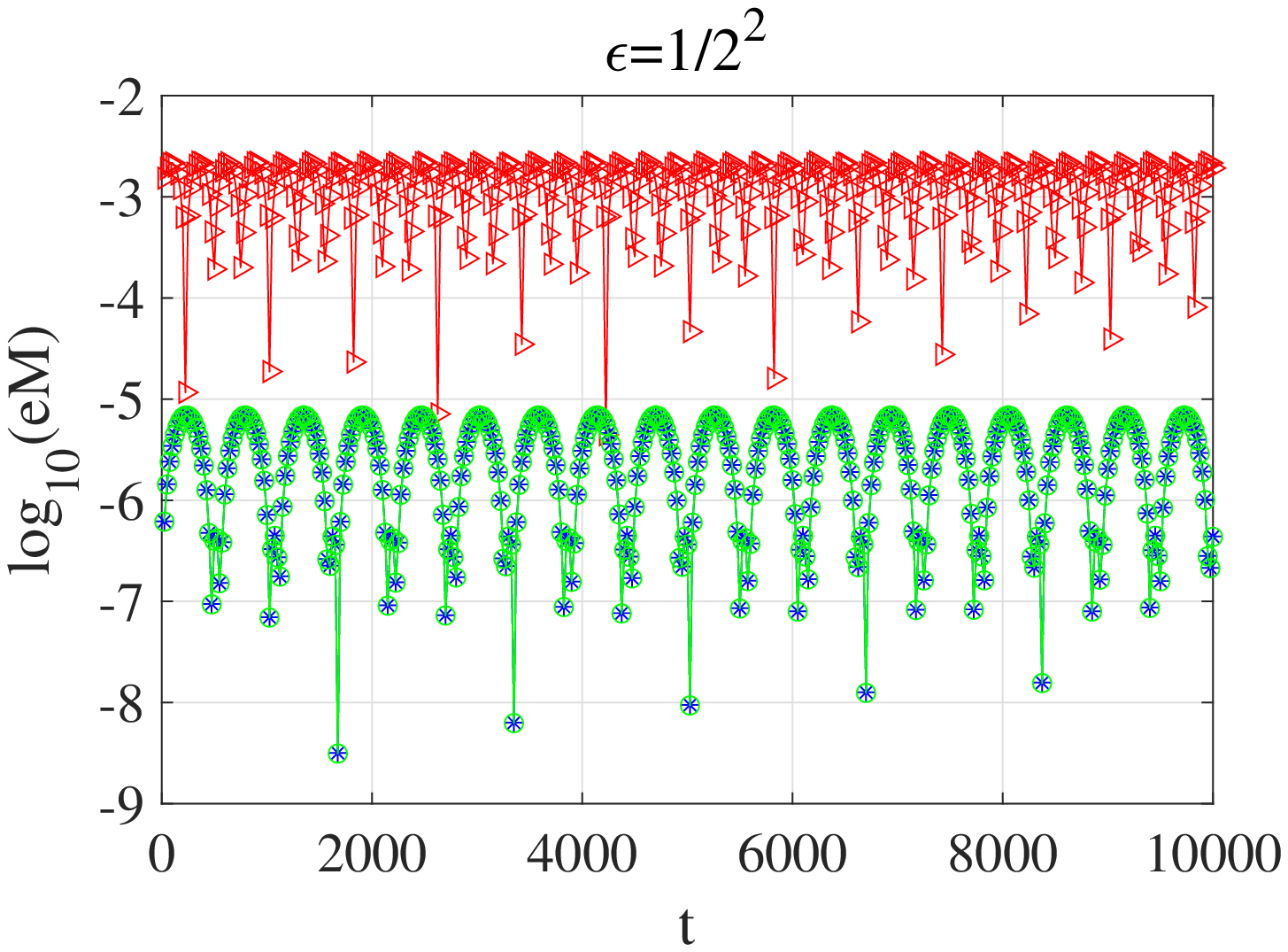}}
			\subfigure{\includegraphics[width=4.7cm,height=4.2cm]{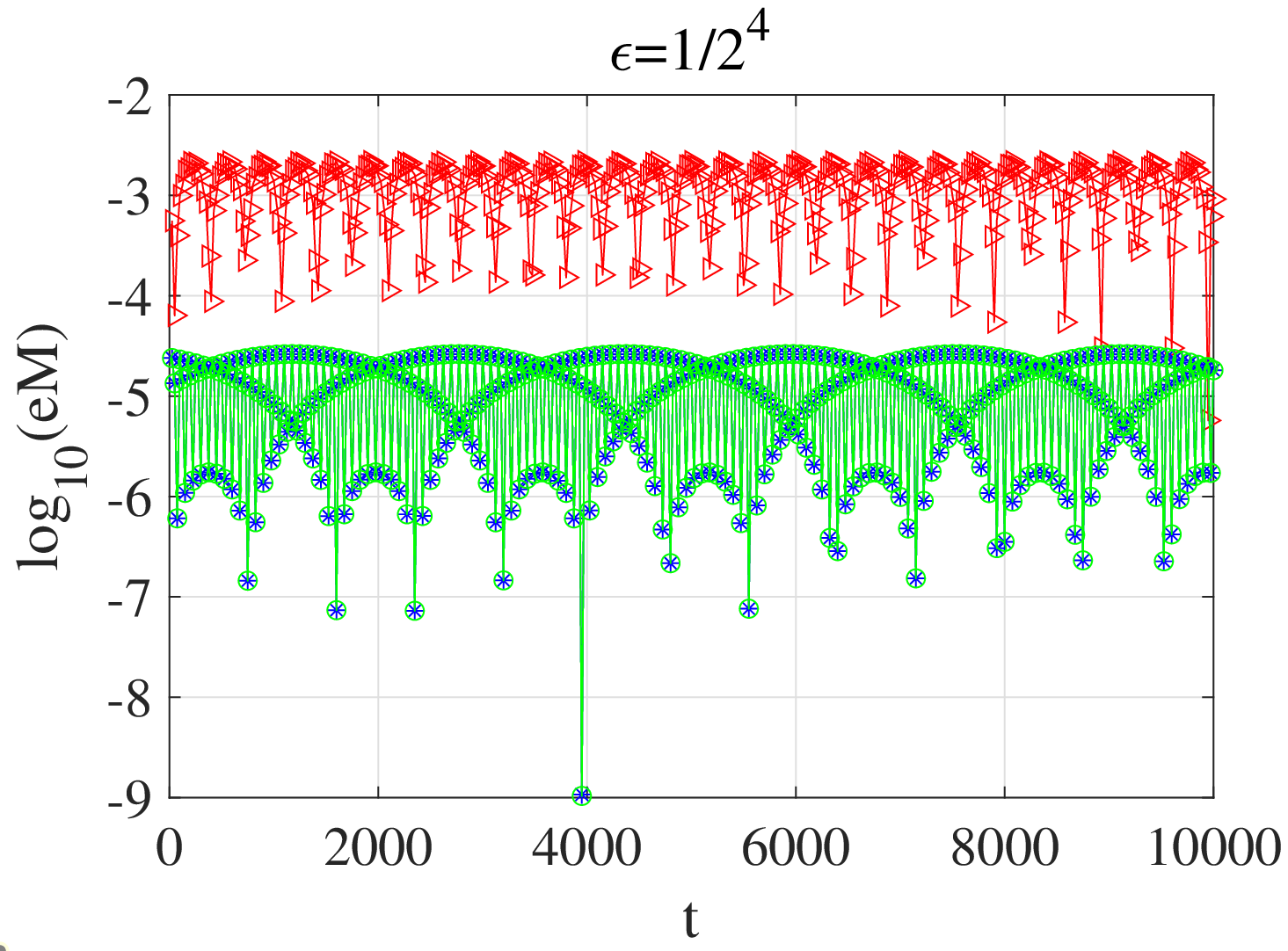}}		
		\end{tabular}
		\caption{{Problem 1.} Evolution of the energy error $e_M$ as function of time $t$. }\label{fig-eM3}
	\end{figure}

	\begin{figure}[t!]
		\centering
		\begin{tabular}[c]{ccc}%
			\subfigure{\includegraphics[width=4.7cm,height=4.2cm]{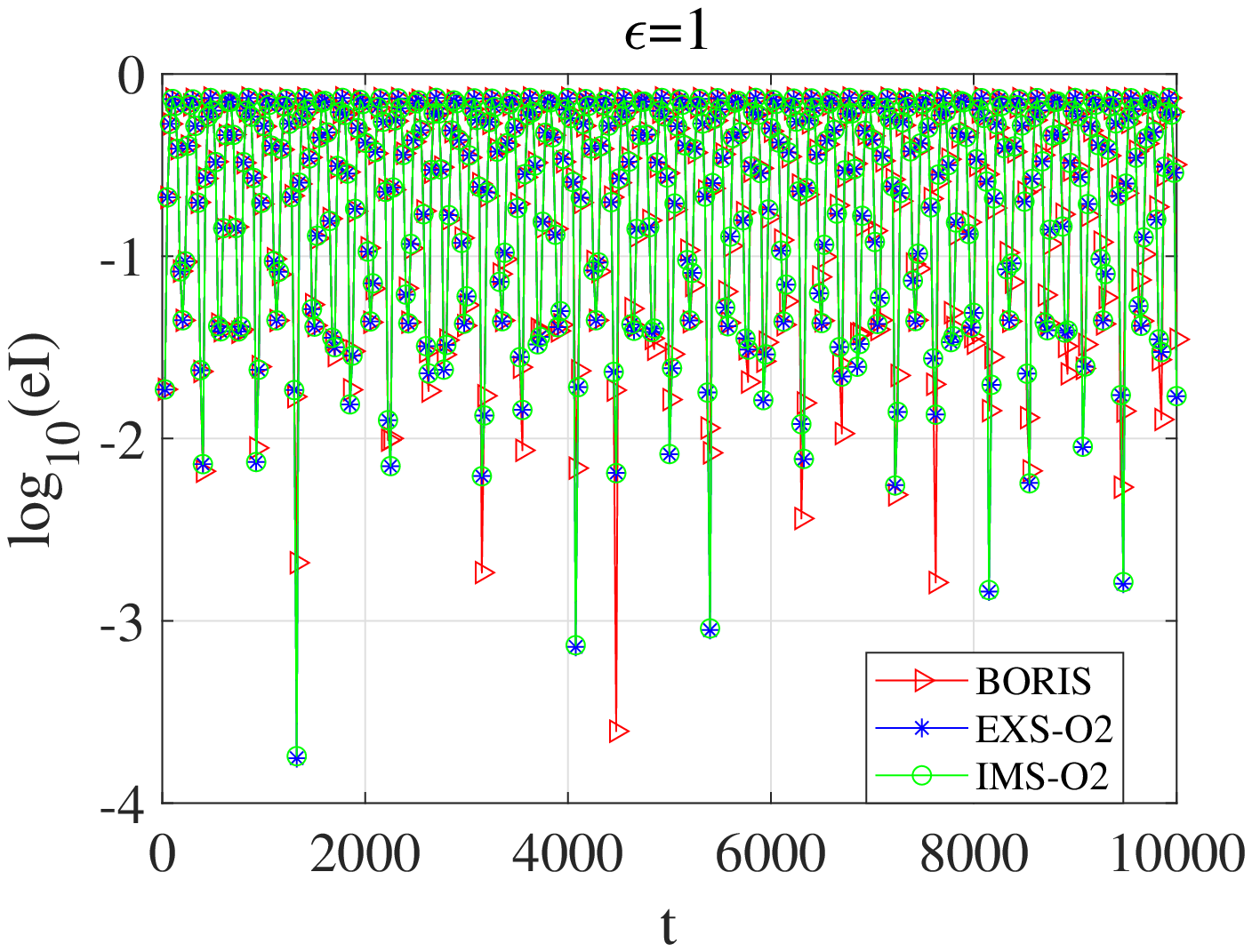}}			\subfigure{\includegraphics[width=4.7cm,height=4.2cm]{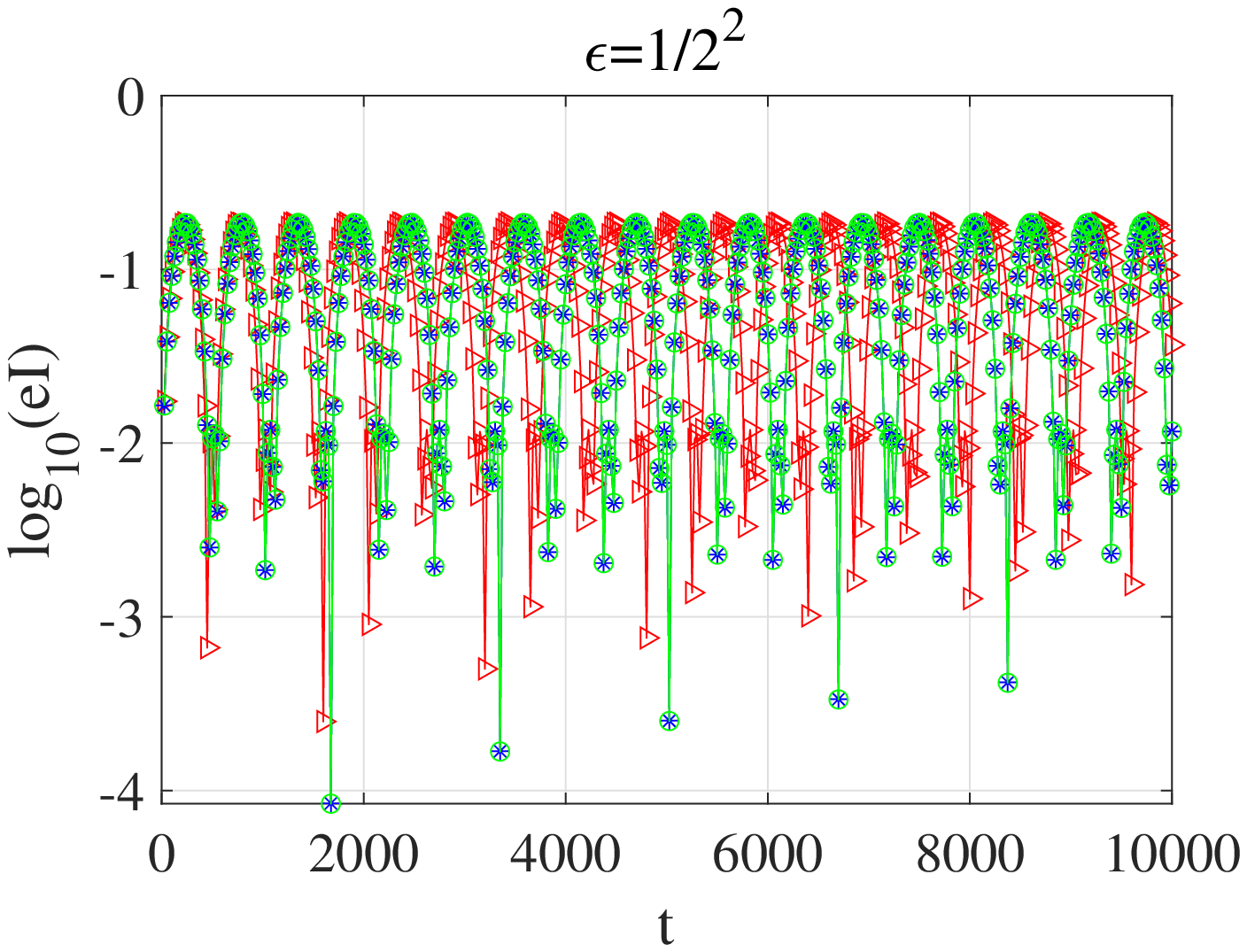}}
			\subfigure{\includegraphics[width=4.7cm,height=4.2cm]{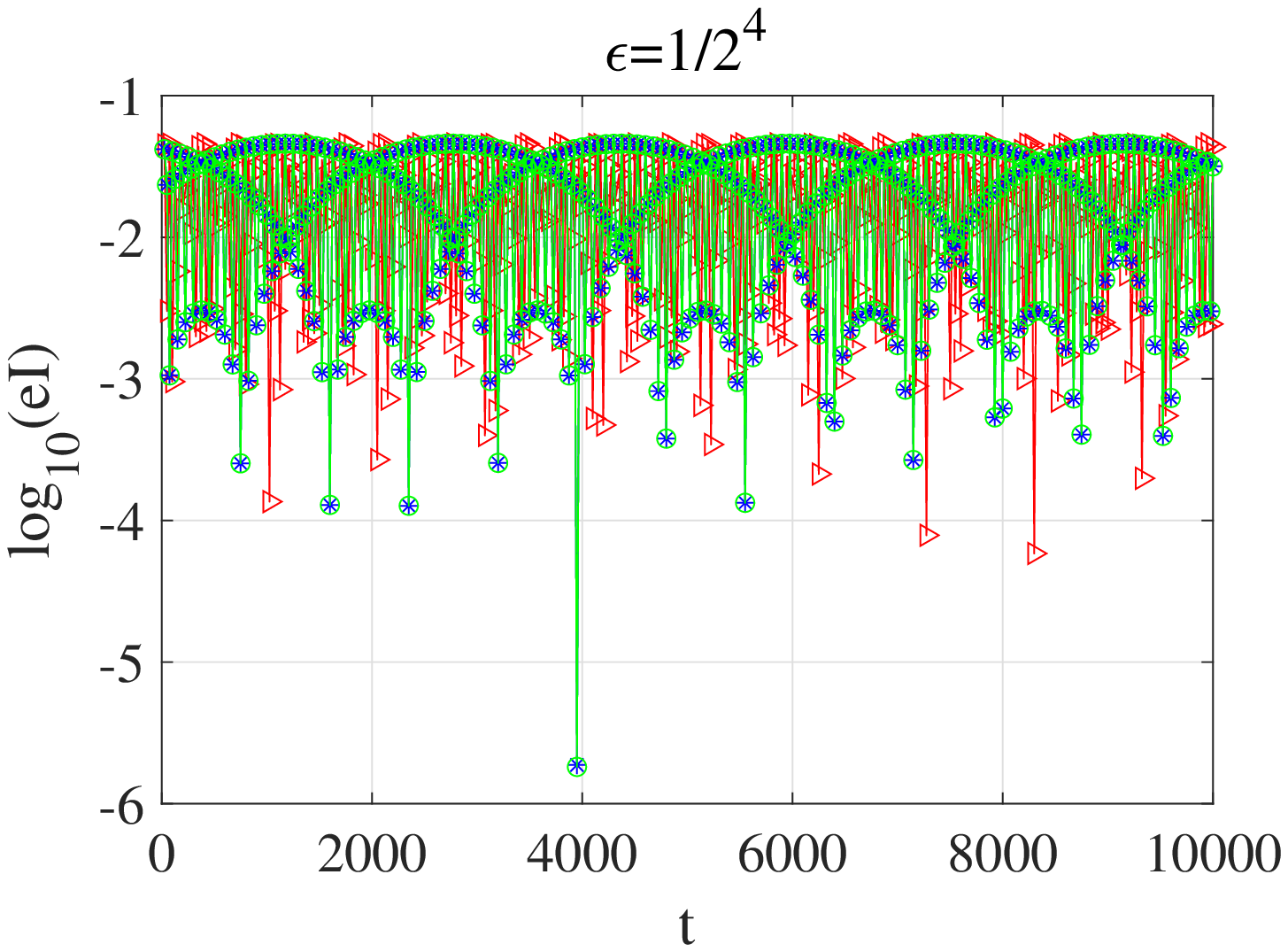}}		
		\end{tabular}
		\caption{{Problem 1.} Evolution of the energy error $e_I$ as function of time $t$. }\label{fig-eI3}
	\end{figure}
	
	\begin{problem}\textbf{({Quadratic scalar potential and constant magnetic field})}\label{prob1}
		We {first consider} the charged particle dynamics \eqref{CPD} with the quadratic scalar potential $U(x)=\frac{x_1^2+x_2^2+x_3^2}{100}$
		and {the constant magnetic field $B=-\nabla\times\frac{1}{2\eps}(x_2,-x_1,0)^\intercal={\frac{1}{\eps}}(0,0,1)^\intercal.$ {Particularly} when $\eps=1$, {we have} $\widehat{B}=\frac{\tilde{B}}{\abs{B}}=\tilde{B}=S$ and $QS=SQ$. The initial values are chosen as $x(0)=(0,1,0.1)^{\intercal}$ and
	$v(0)=(0.09,0.05,0.20)^{\intercal}$.} The errors in \eqref{error} are respectively displayed in {Figs.} \ref{fig-eH3}--\ref{fig-eI3}. In addition, {let  $e_{H_h}:=\frac{\abs{H_h(x^n,v^n)-H_h(x^0,v^0)}}{\abs{H_h(x^0,v^0)}}$ with $H_h$ \eqref{Hh}. The conservation of $e_{H_h}$ by EXS-O2 is displayed in Fig.} \ref{fig-eHh}.
	\end{problem}
	\begin{figure}[t!]
		\centering
		\begin{tabular}[c]{ccc}%
			\subfigure{\includegraphics[width=4.7cm,height=4.2cm]{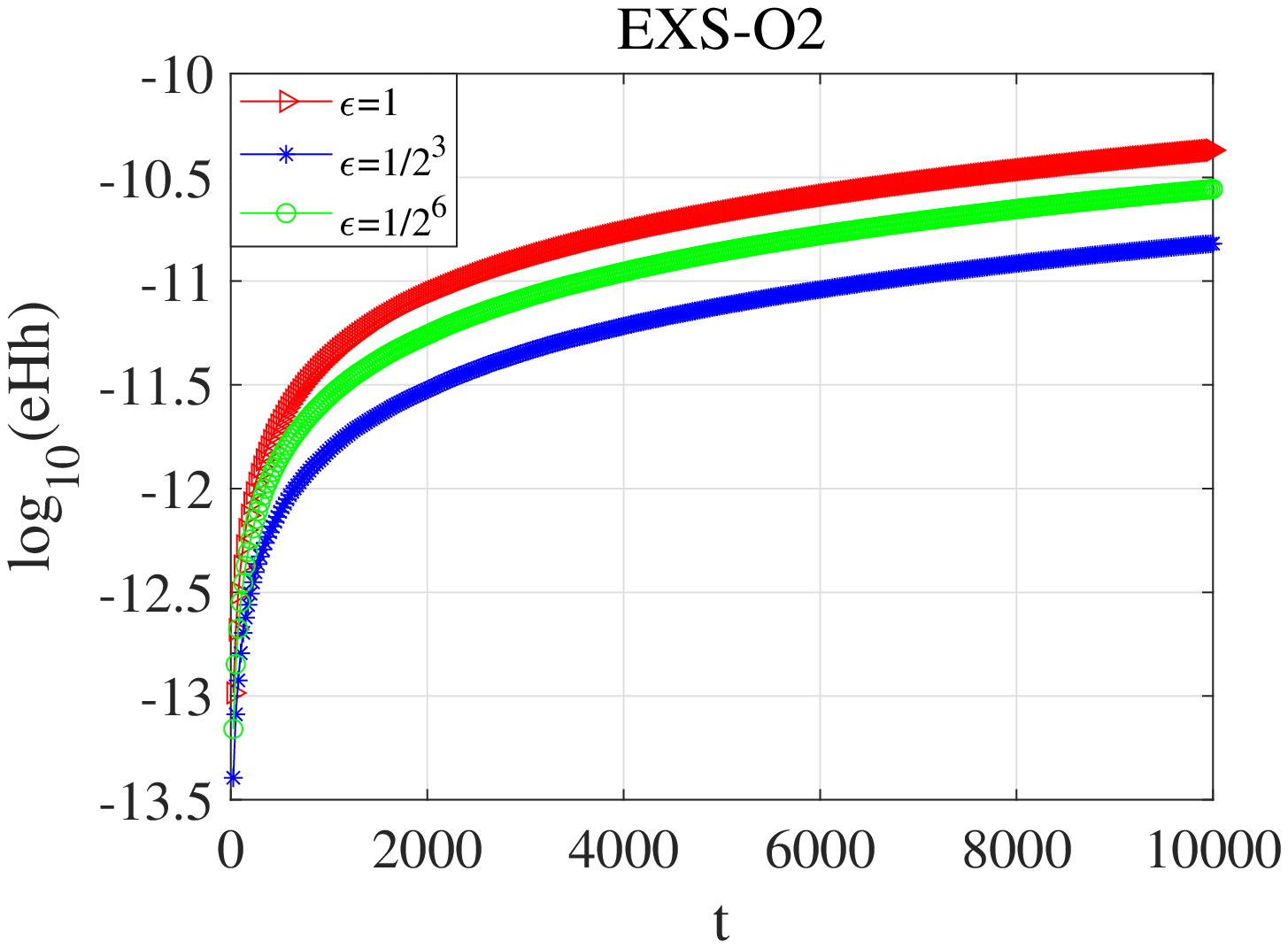}}			\subfigure{\includegraphics[width=4.7cm,height=4.2cm]{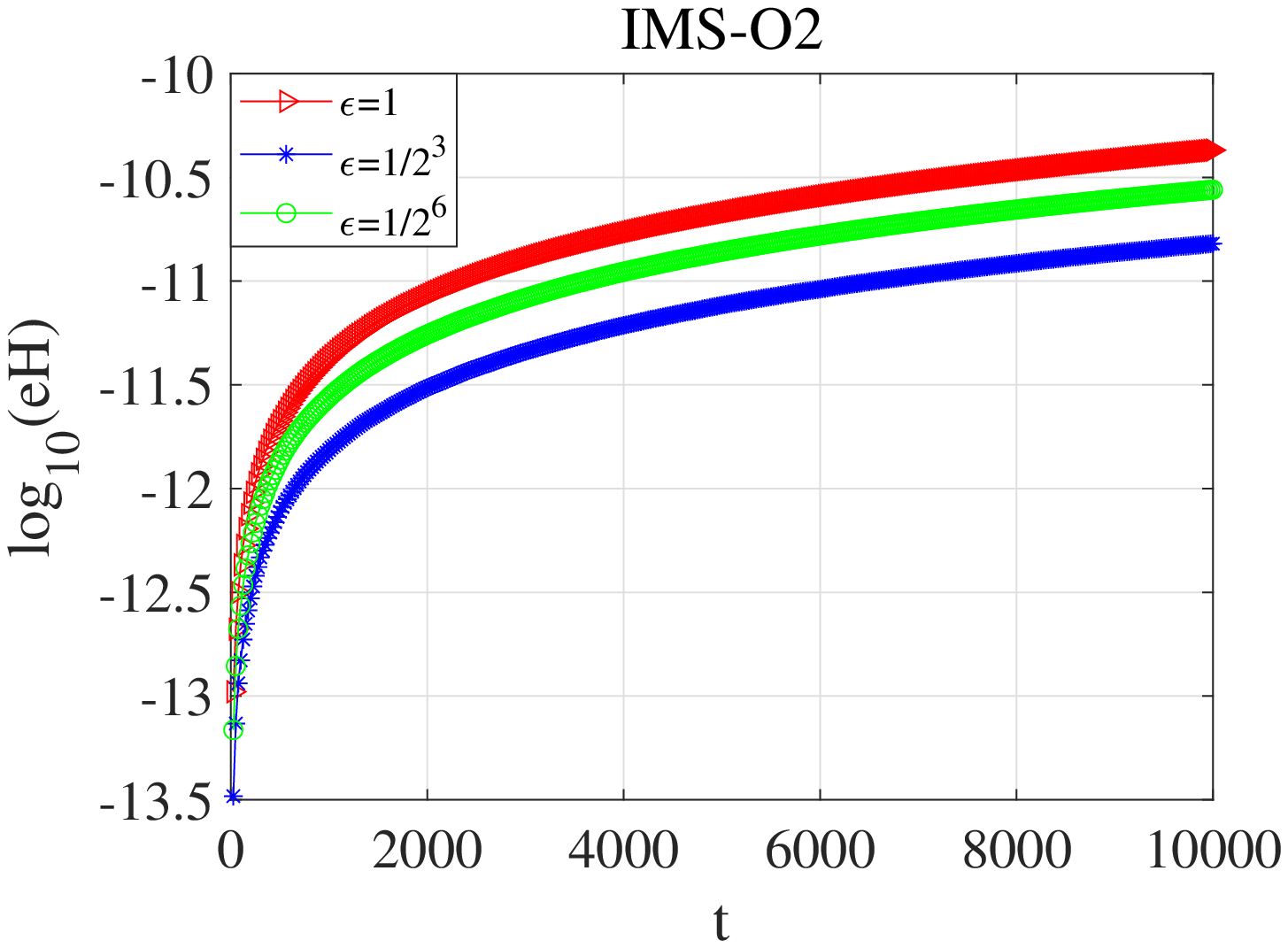}}		
		\end{tabular}
		\caption{{Problem 1.} Evolution of the energy error $e_{H_h}$ for EXS-O2 and $e_H$ for IMS-O2 as functions of time $t$ with $\eps=1,1/2^3,1/2^6$. }\label{fig-eHh}
	\end{figure}

	\begin{figure}[t!]
	\centering
	\begin{tabular}[c]{ccc}%
		\subfigure{\includegraphics[width=4.7cm,height=4.2cm]{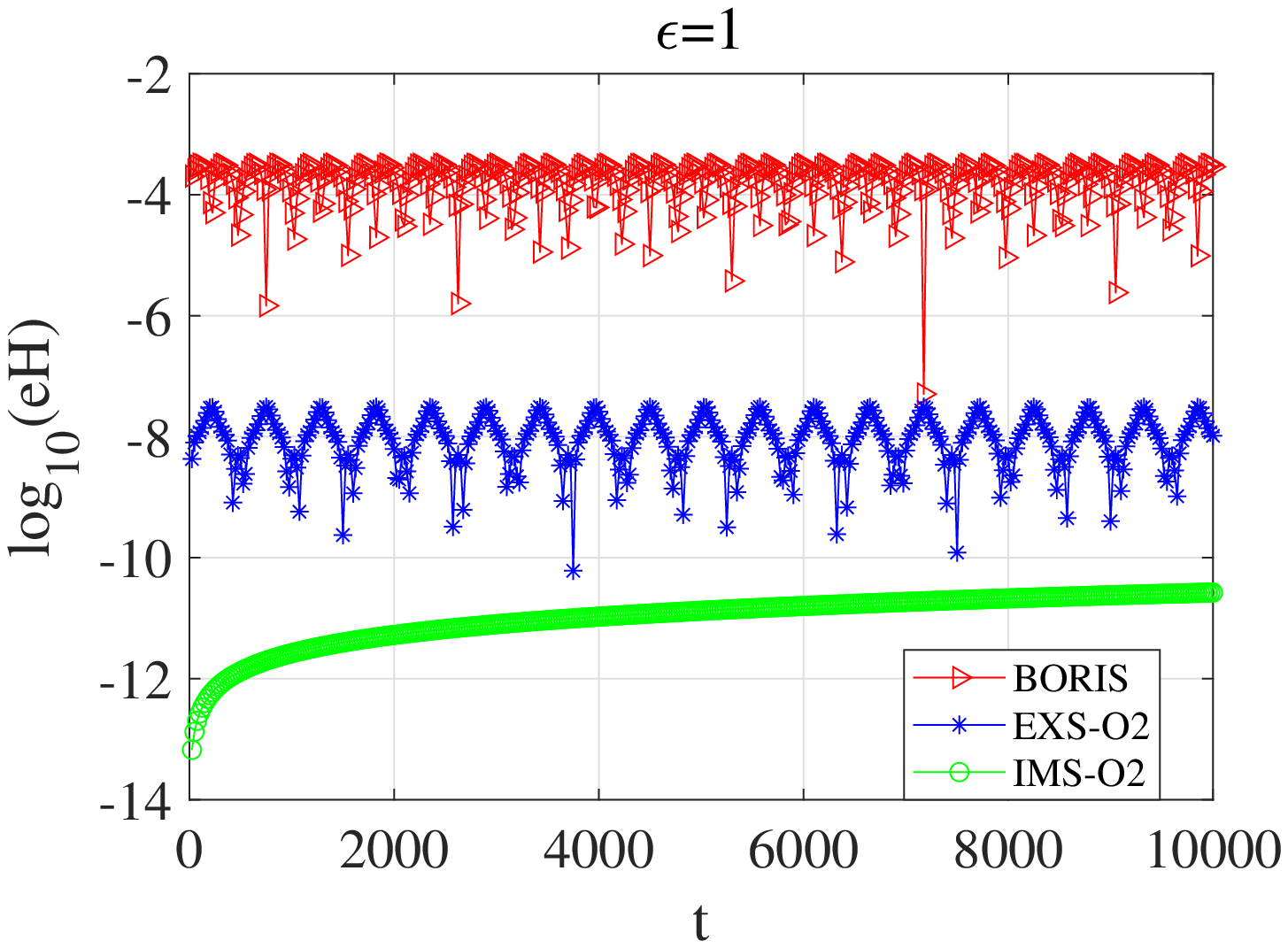}}			\subfigure{\includegraphics[width=4.7cm,height=4.2cm]{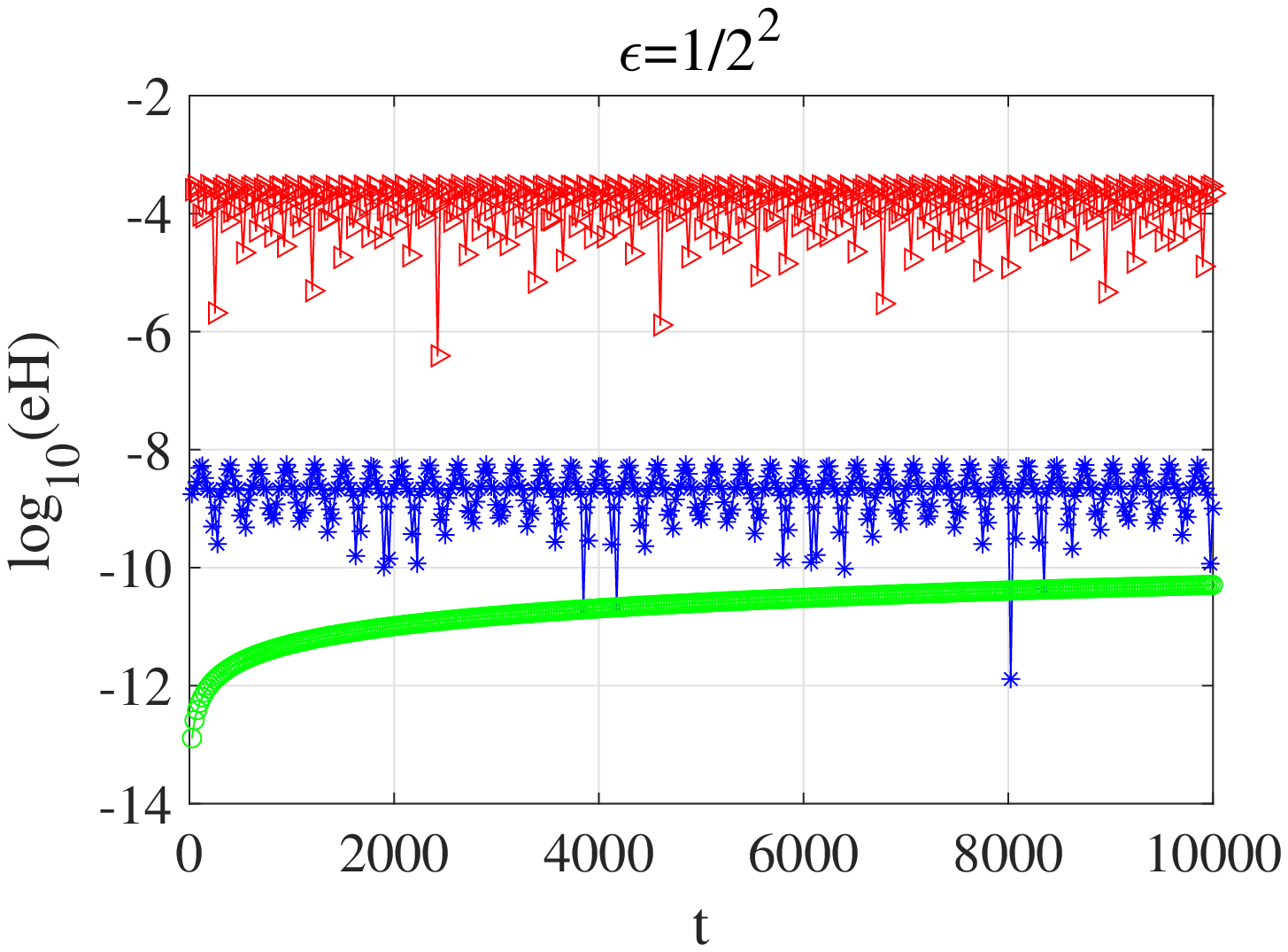}}
		\subfigure{\includegraphics[width=4.7cm,height=4.2cm]{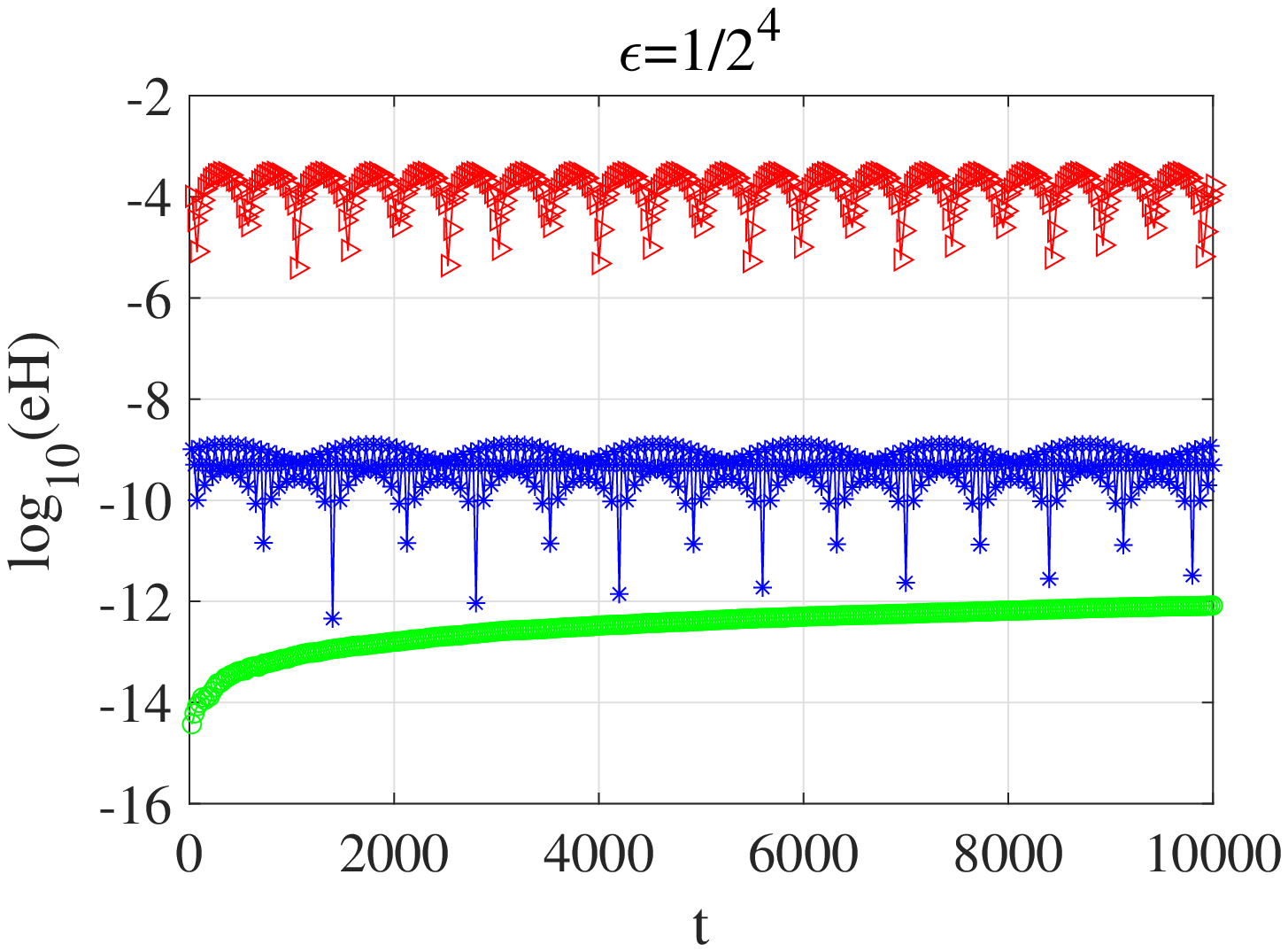}}		
	\end{tabular}
	\caption{{Problem 2.} Evolution of the energy error $e_H$ as function of time $t$. }\label{fig-eH2}
\end{figure}

    \begin{problem}\textbf{(Constant magnetic field)}\label{prob2}
    {As the second} problem, we consider a general scalar potential
$
		U(x)=\frac{1}{100\sqrt{x_1^2+x_2^2}}.$
The magnetic field and the initial values are the same as {those} in Problem \ref{prob1}. The errors in \eqref{error} are respectively {presented} in {Figs.} \ref{fig-eH2}$-$\ref{fig-eI2}.
    \end{problem}

	\begin{problem}\textbf{(General case)}\label{prob3}
		In the third numerical experiment, we {are interested in} a {general case that} the scalar potential is {the same as it in Problem \ref{prob2}}
		and the magnetic field is
		\begin{equation*}
			B(x)=\nabla\times\frac{1}{3\eps}(-x_2\sqrt{x_1^2+x_2^2},-x_1\sqrt{x_1^2+x_2^2},0)^\intercal={\frac{1}{\eps}}(0,0,\sqrt{x_1^2+x_2^2})^\intercal.
		\end{equation*}
		{The initial values are the same as {those} in Problem \ref{prob1}.} Then errors in \eqref{error} are {severally shown} in {Figs.} \ref{fig-eH1}$-$\ref{fig-eI1}. {Besides, in order to show the accuracy of the methods, we} plot the global errors $error:=\frac{\abs{x^n-x(t_n)}}{x(t_n)}+\frac{\abs{v^n-v(t_n)}}{v(t_n)}$ at $t_{end}=1$ in Fig \ref{fig-error}, where the reference solution is obtained by using “ode45” of MATLAB.
	\end{problem}

{In accordance with} these numerical results, we have the following observations.

\begin{figure}[t!]
	\centering
	\begin{tabular}[c]{ccc}%
		\subfigure{\includegraphics[width=4.7cm,height=4.2cm]{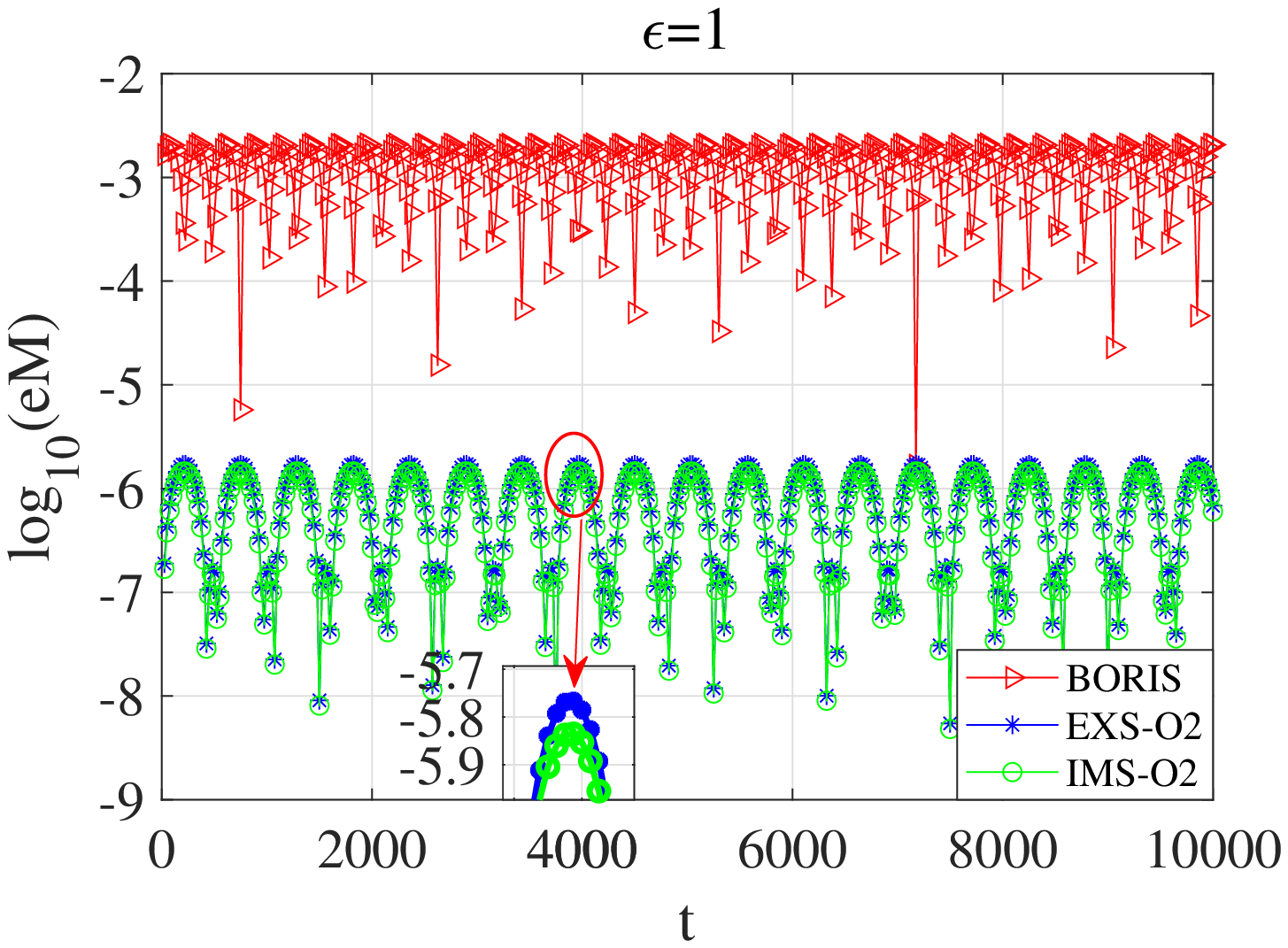}}			\subfigure{\includegraphics[width=4.7cm,height=4.2cm]{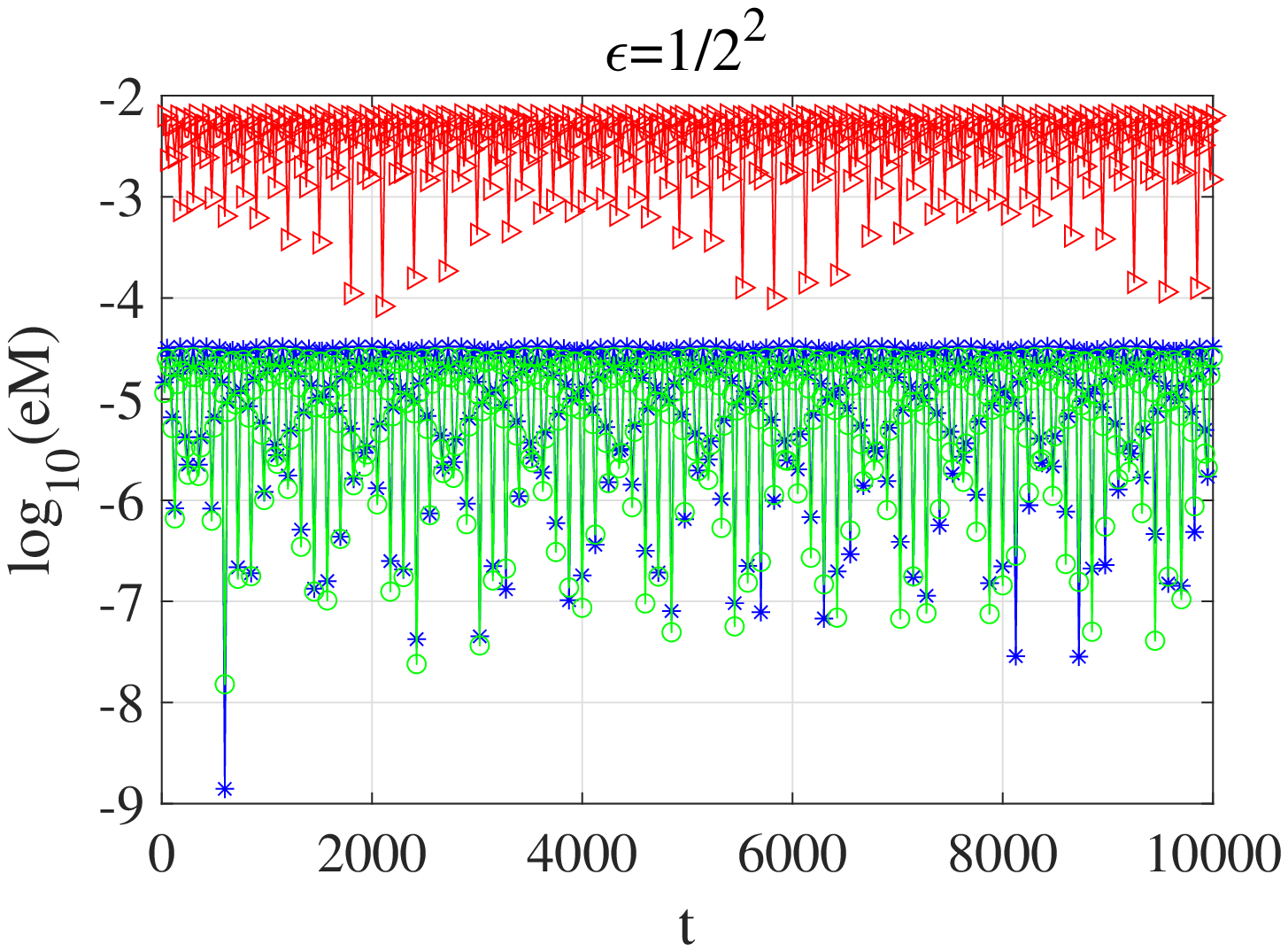}}
		\subfigure{\includegraphics[width=4.7cm,height=4.2cm]{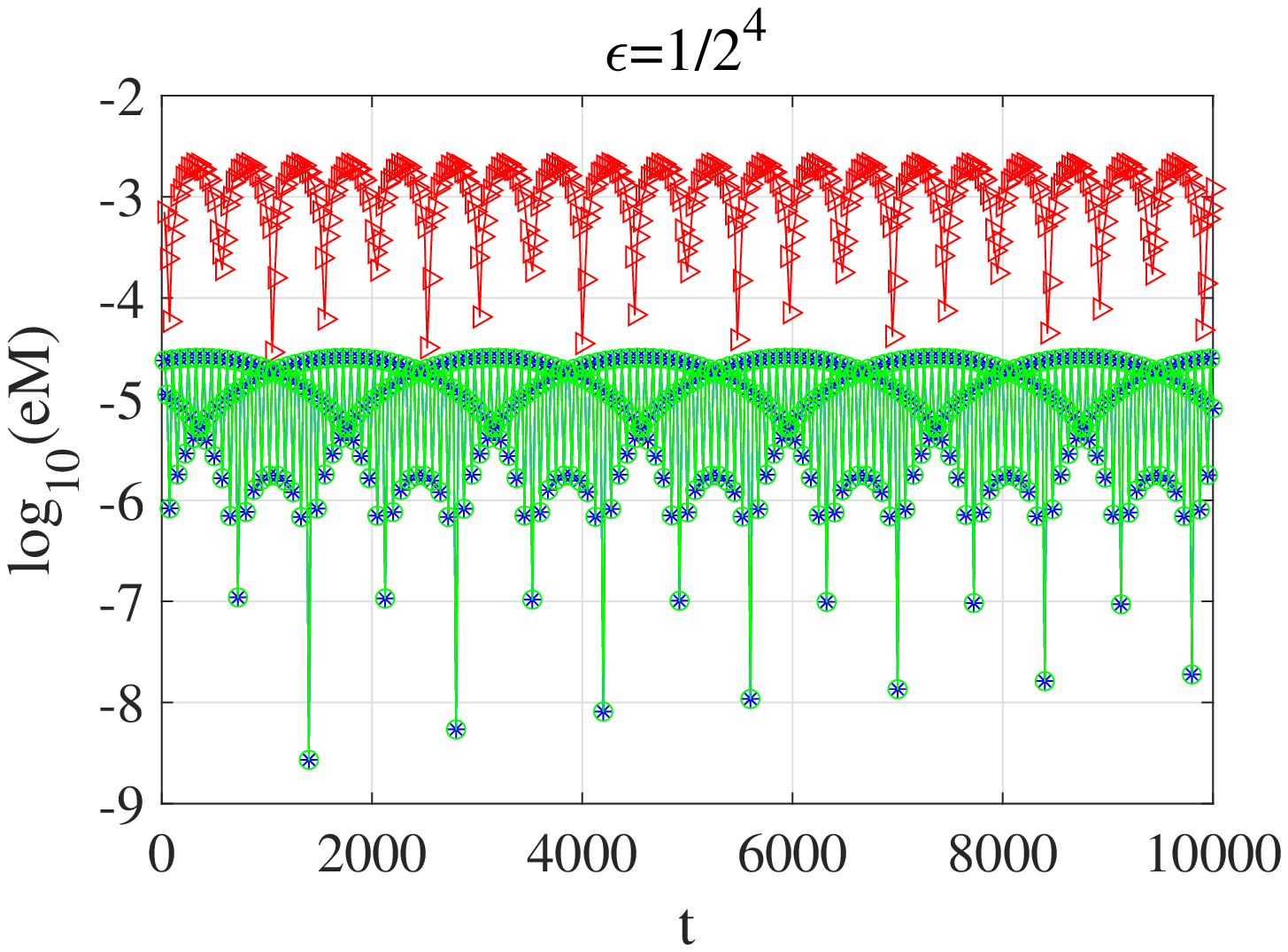}}		
	\end{tabular}
	\caption{{Problem 2.} Evolution of the energy error $e_M$ as function of time $t$. }\label{fig-eM2}
\end{figure}
\begin{figure}[t!]
	\centering
	\begin{tabular}[c]{ccc}%
		\subfigure{\includegraphics[width=4.7cm,height=4.2cm]{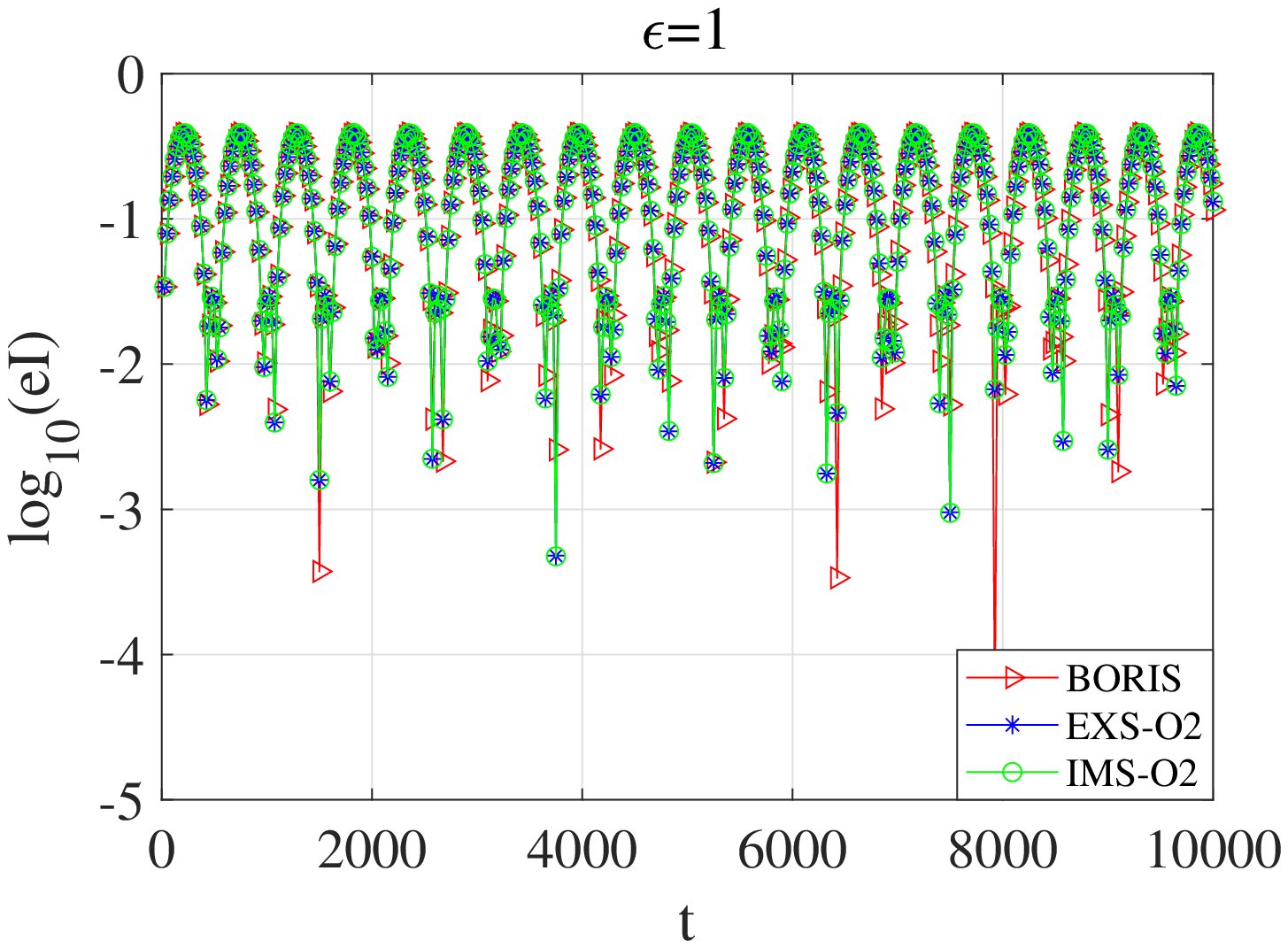}}			\subfigure{\includegraphics[width=4.7cm,height=4.2cm]{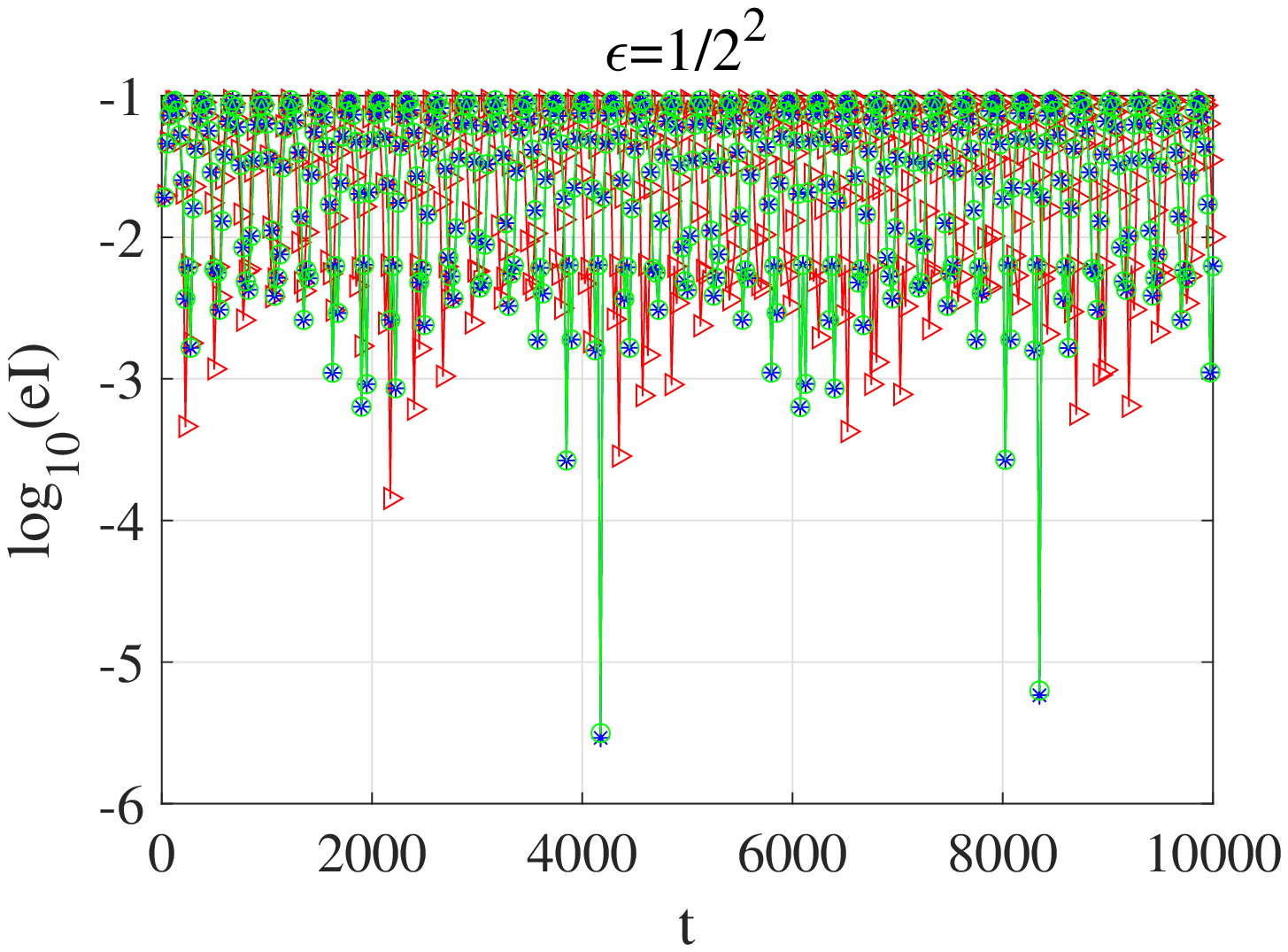}}
		\subfigure{\includegraphics[width=4.7cm,height=4.2cm]{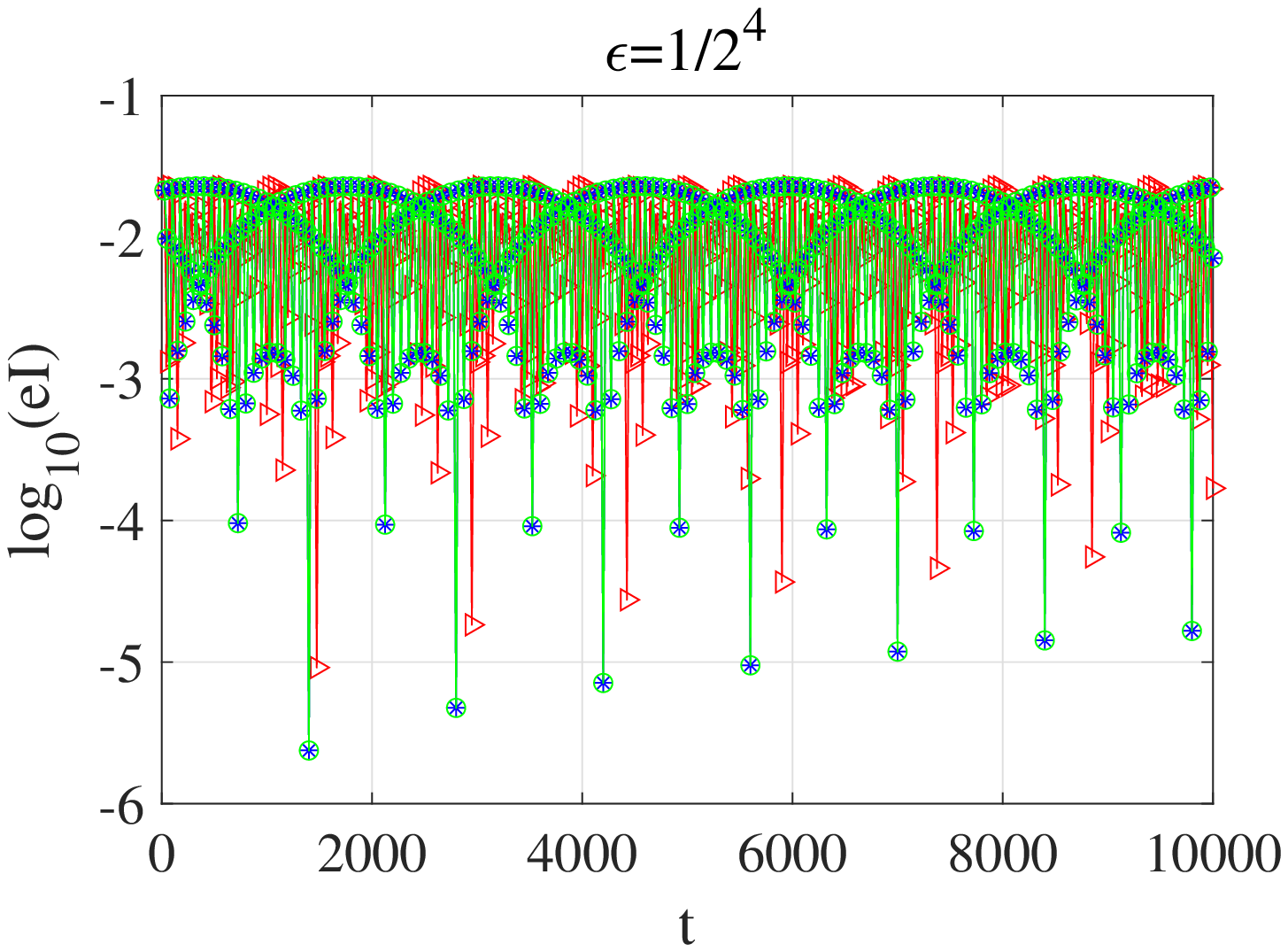}}		
	\end{tabular}
	\caption{{Problem 2.} Evolution of the energy error $e_I$ as function of time $t$. }\label{fig-eI2}
\end{figure}

	$\bullet$ \textbf{Energy conservation}.  Clearly from Figs.  \ref{fig-eH3}, \ref{fig-eH2} and \ref{fig-eH1},  we can see that  only IMS-O2 is energy-preserving,  Boris and EXS-O2 show a long-term energy behavior, and EXS-O2 behaves better than Boris. Meanwhile, Fig \ref{fig-eHh} demonstrates that  EXS-O2 exactly {preserves} the energy $H_h$.
%
%
	
	$\bullet$ \textbf{Momentum conservation}. {In the light of the results given by Figs. \ref{fig-eM3}, \ref{fig-eM2} and \ref{fig-eM1},  all the three methods are not momentum-preserving but they {hold} a long time momentum conservation.
 Besides, the behaviour of our methods {is better} than the Boris algorithm. }
	
	\begin{figure}[t!]
	\centering
	\begin{tabular}[c]{ccc}%
		\subfigure{\includegraphics[width=4.7cm,height=4.2cm]{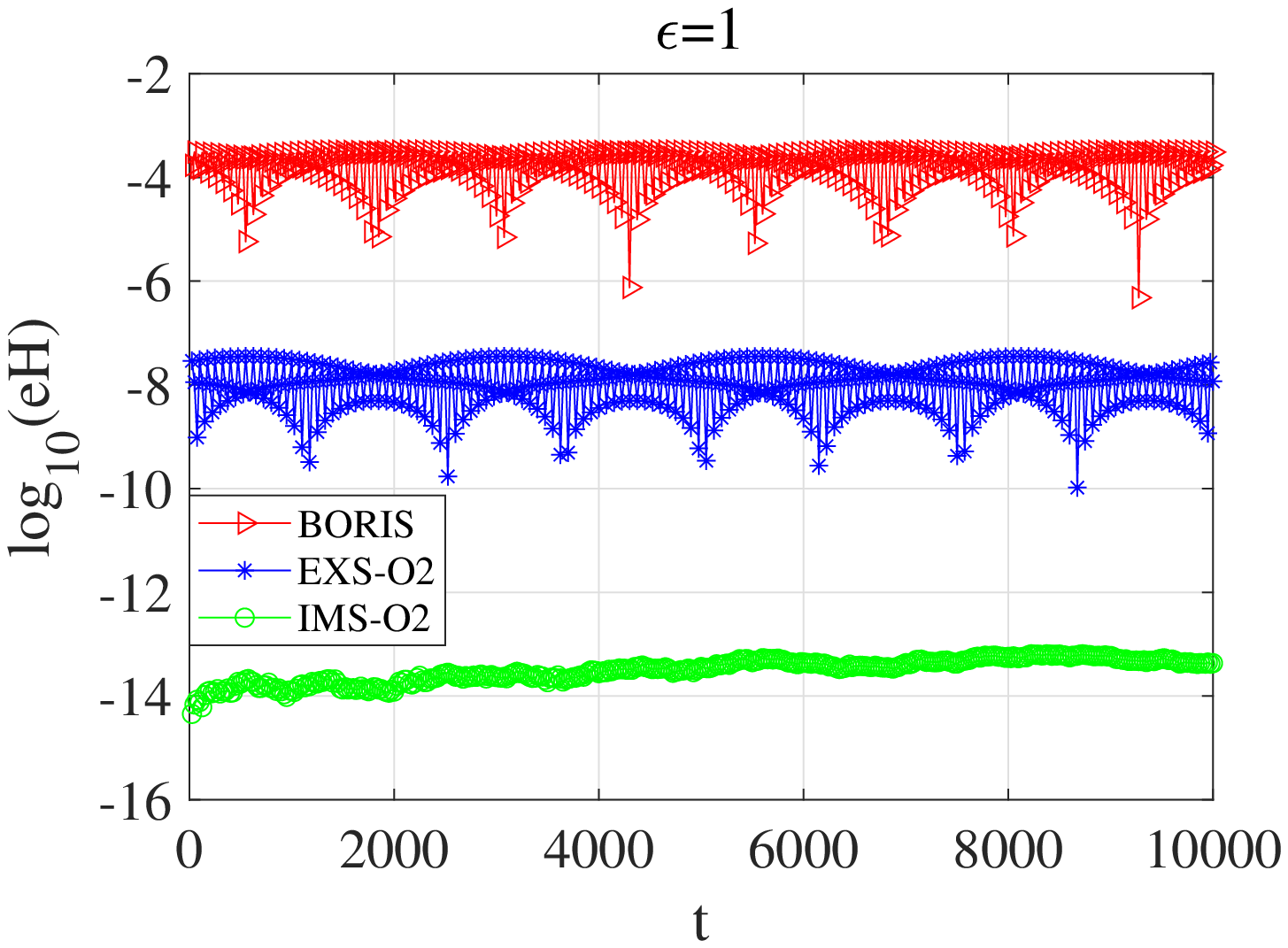}}			\subfigure{\includegraphics[width=4.7cm,height=4.2cm]{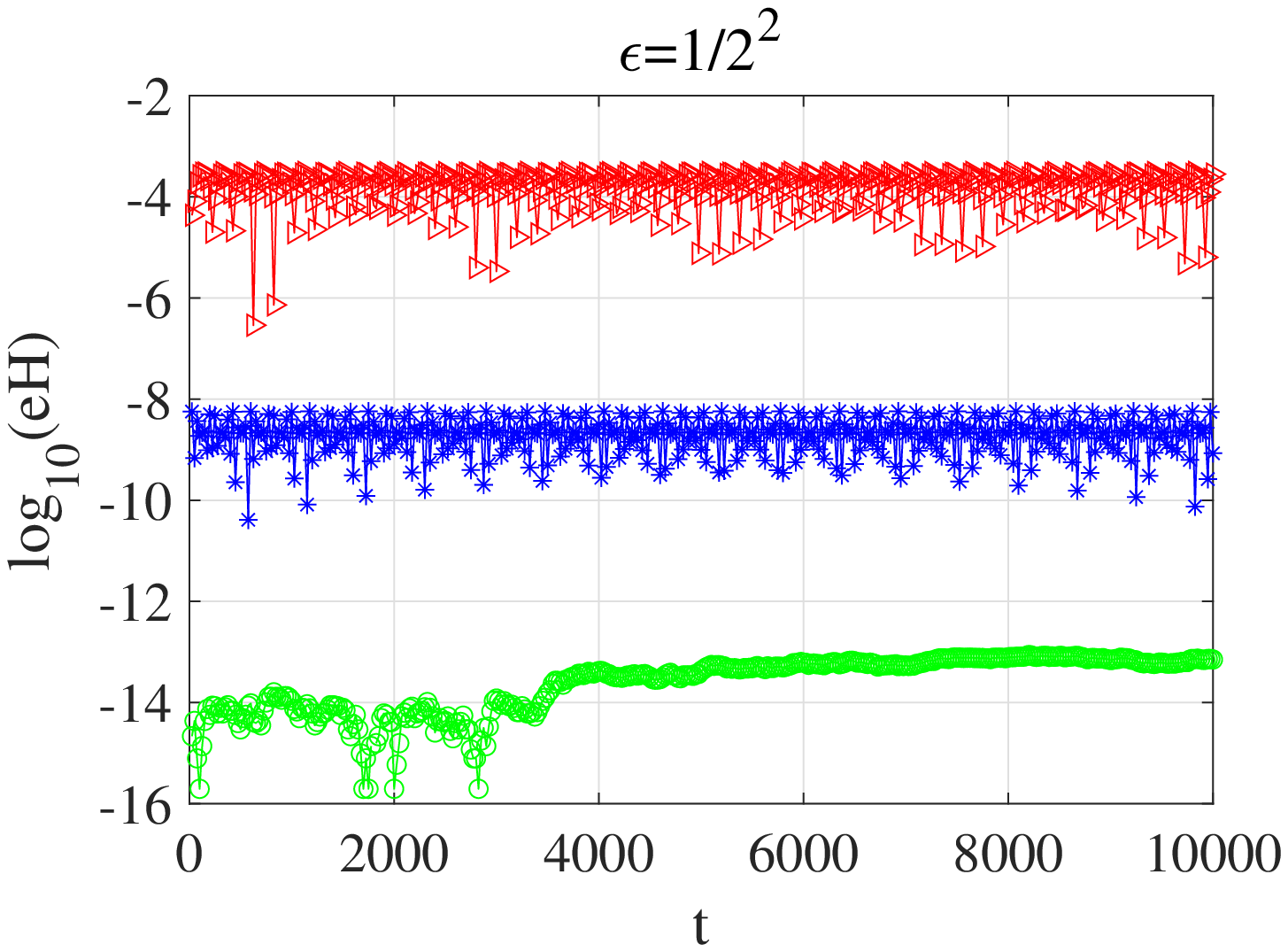}}
		\subfigure{\includegraphics[width=4.7cm,height=4.2cm]{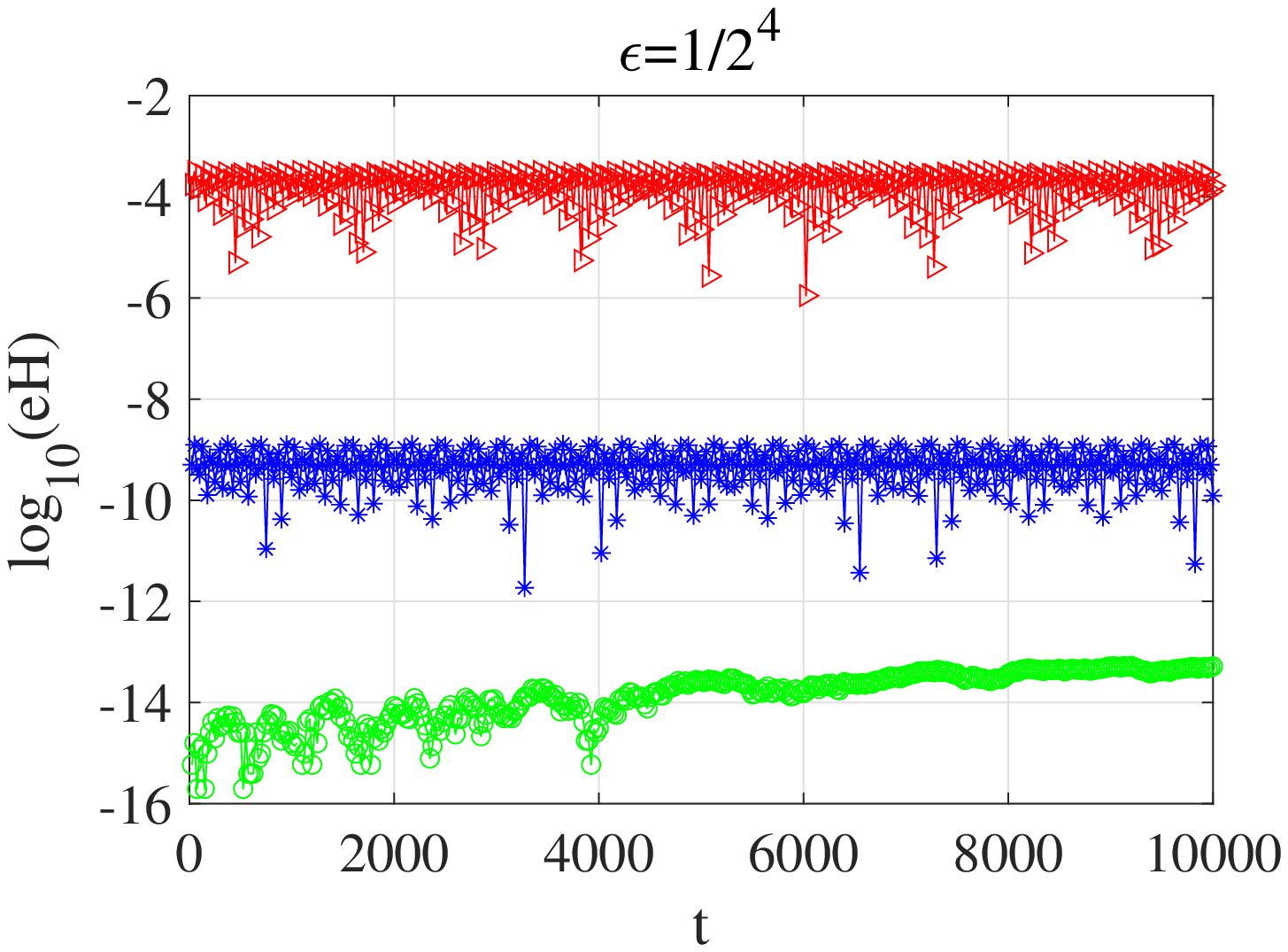}}		
	\end{tabular}
	\caption{{Problem 3.} Evolution of the energy error $e_H$ as function of time $t$. }\label{fig-eH1}
\end{figure}
\begin{figure}[t!]
	\centering
	\begin{tabular}[c]{ccc}%
		\subfigure{\includegraphics[width=4.7cm,height=4.2cm]{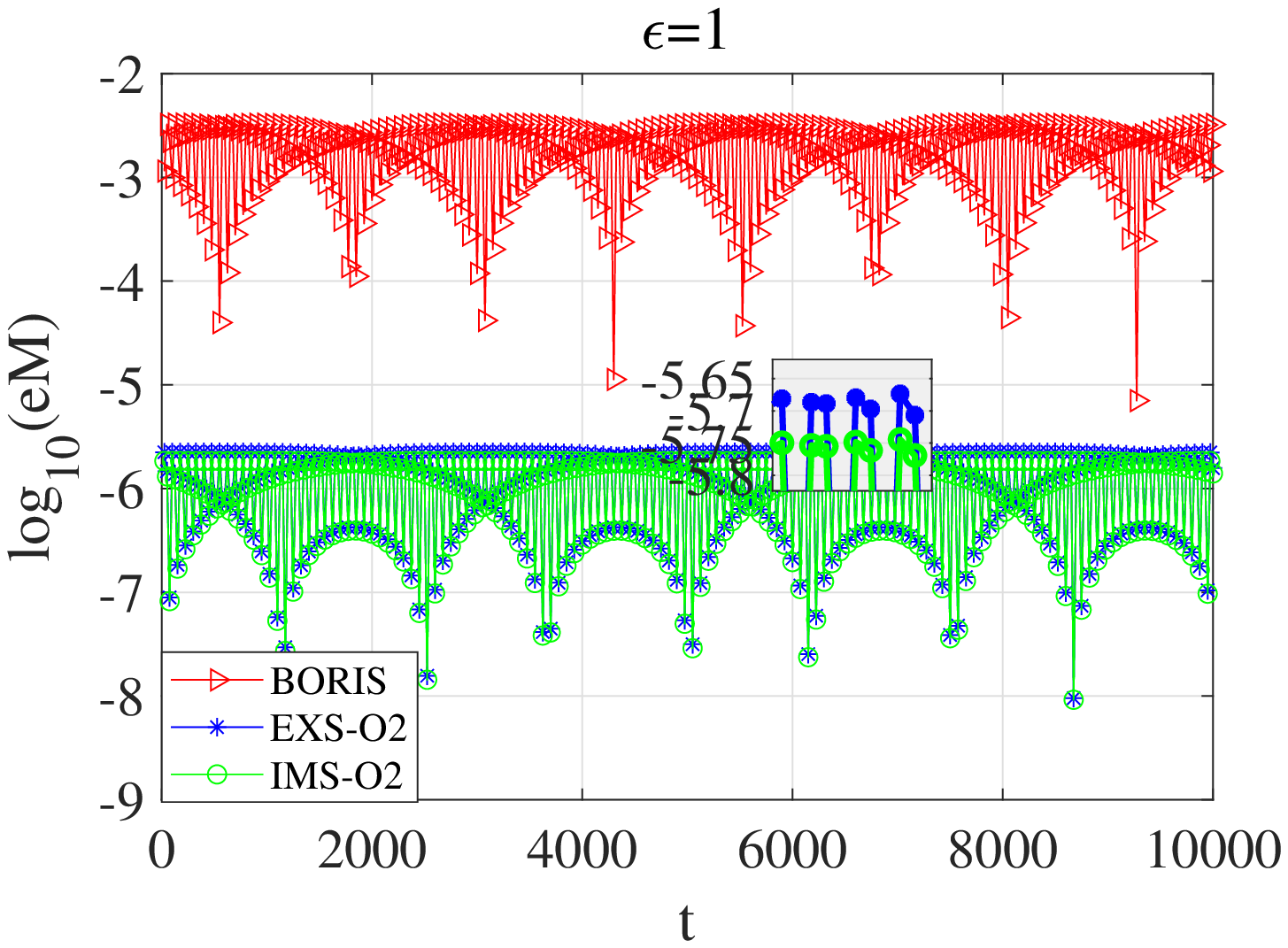}}			\subfigure{\includegraphics[width=4.7cm,height=4.2cm]{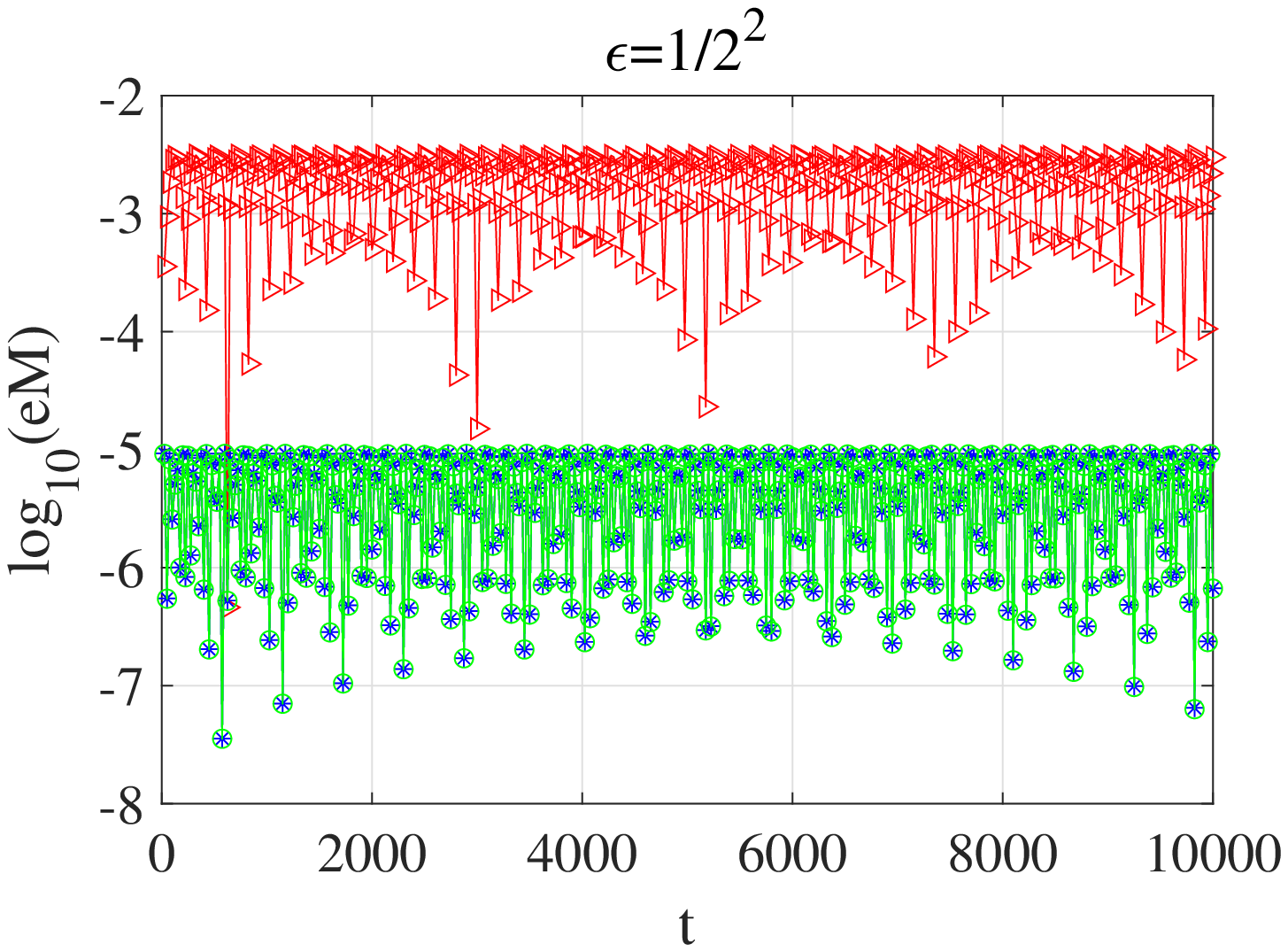}}
		\subfigure{\includegraphics[width=4.7cm,height=4.2cm]{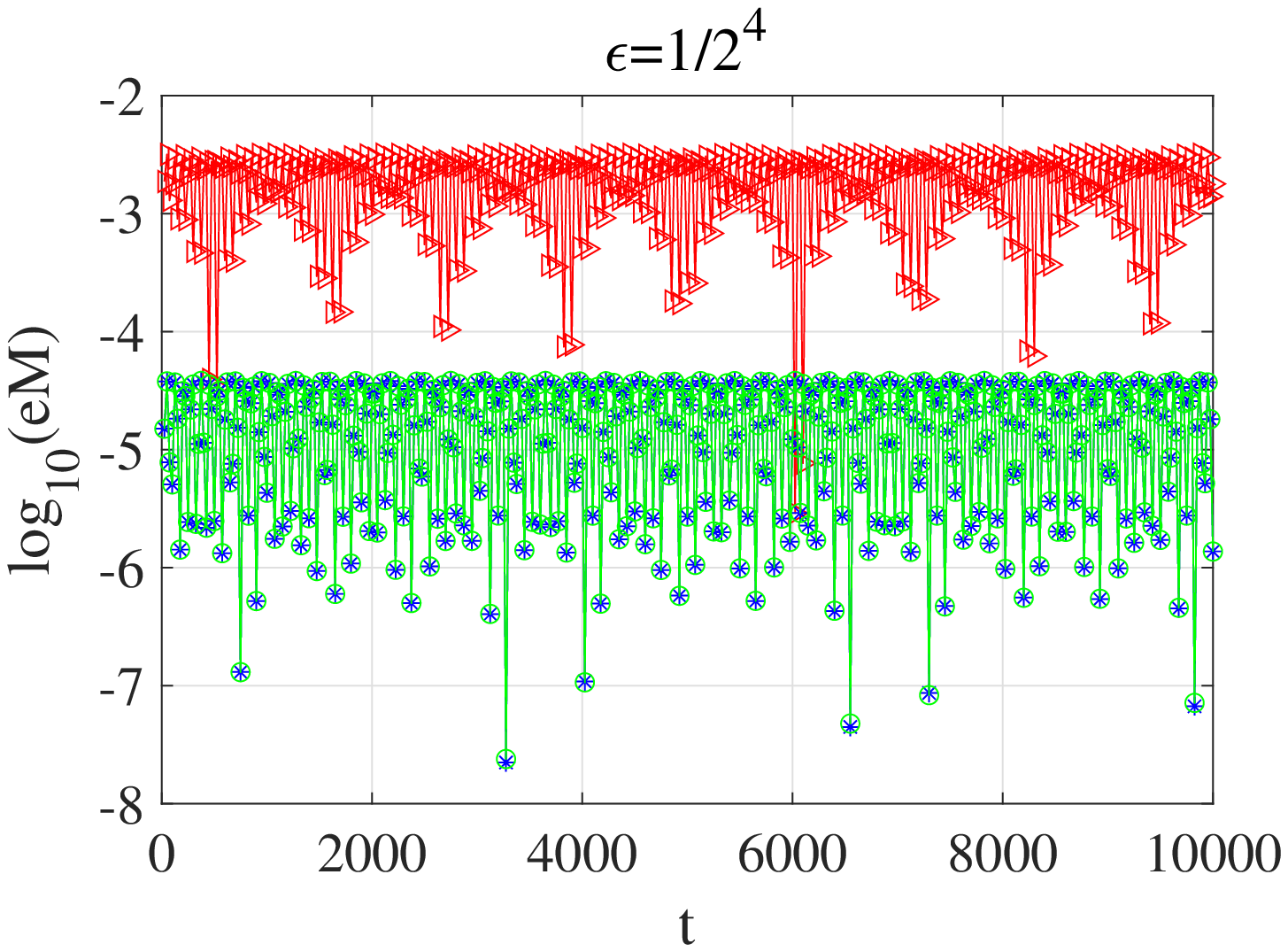}}		
	\end{tabular}
	\caption{{Problem 3.} Evolution of the energy error $e_M$ as function of time $t$. }\label{fig-eM1}
\end{figure}

\begin{figure}[t!]
	\centering
	\begin{tabular}[c]{ccc}%
		\subfigure{\includegraphics[width=4.7cm,height=4.2cm]{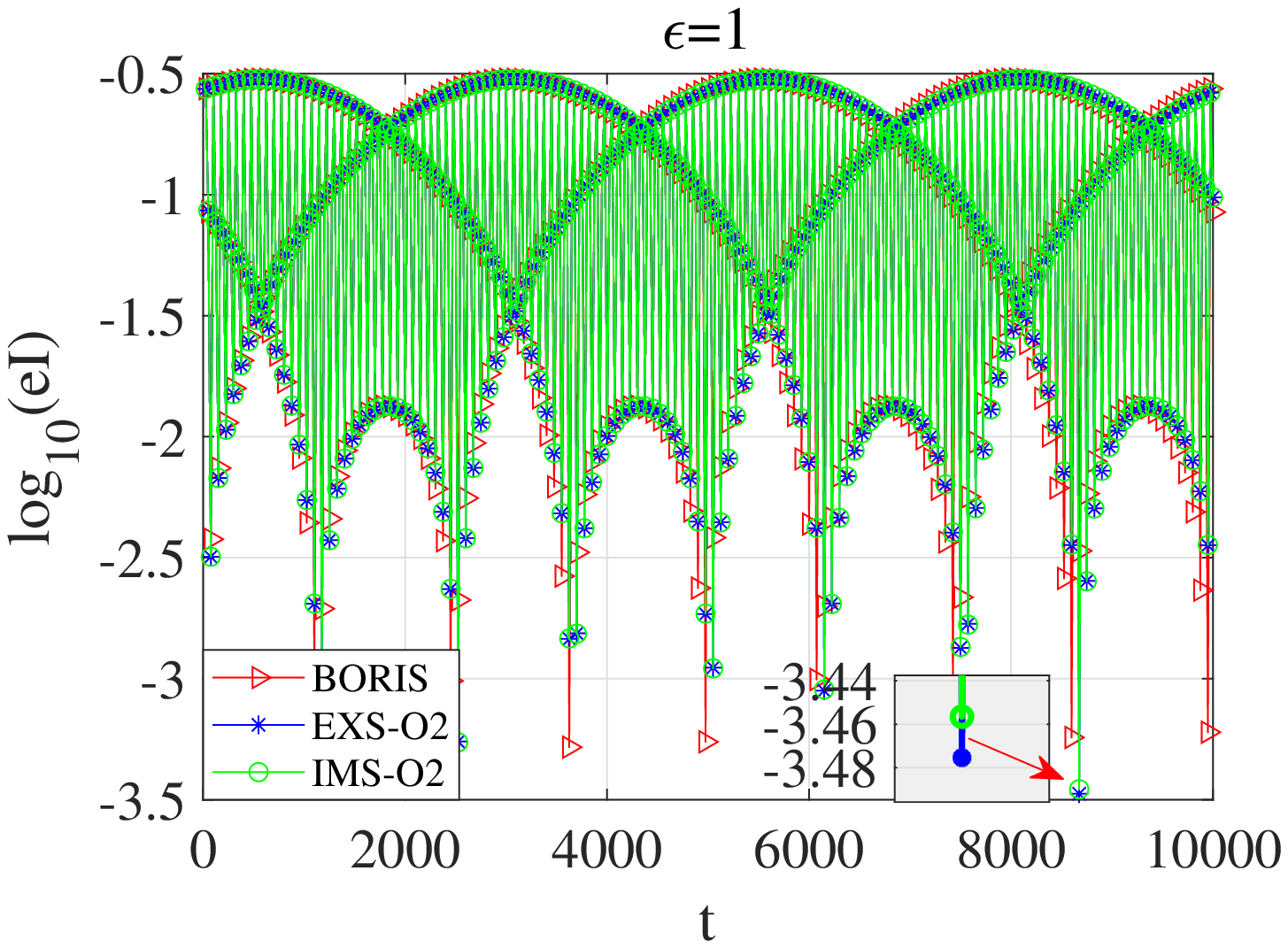}}			\subfigure{\includegraphics[width=4.7cm,height=4.2cm]{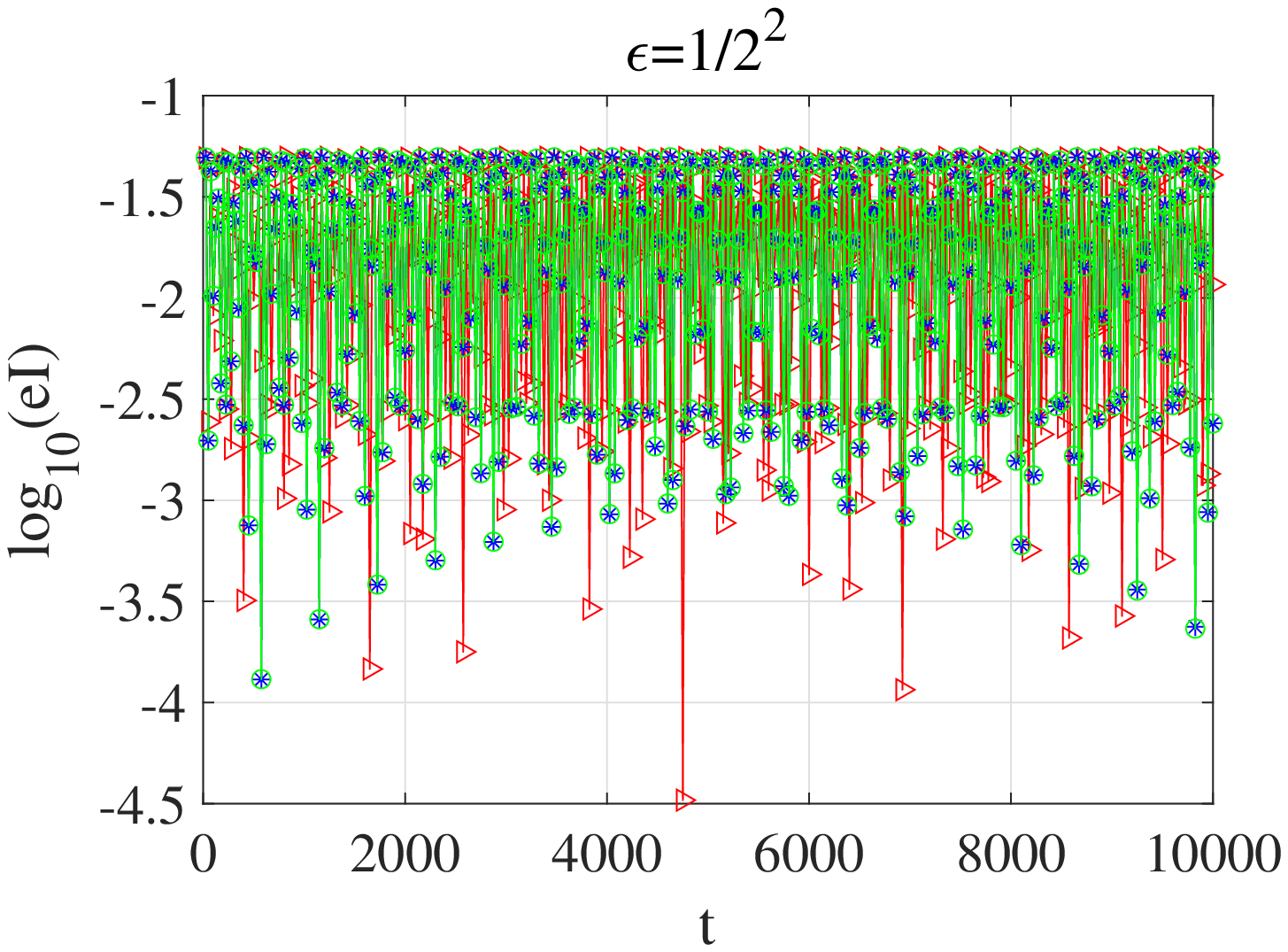}}
		\subfigure{\includegraphics[width=4.7cm,height=4.2cm]{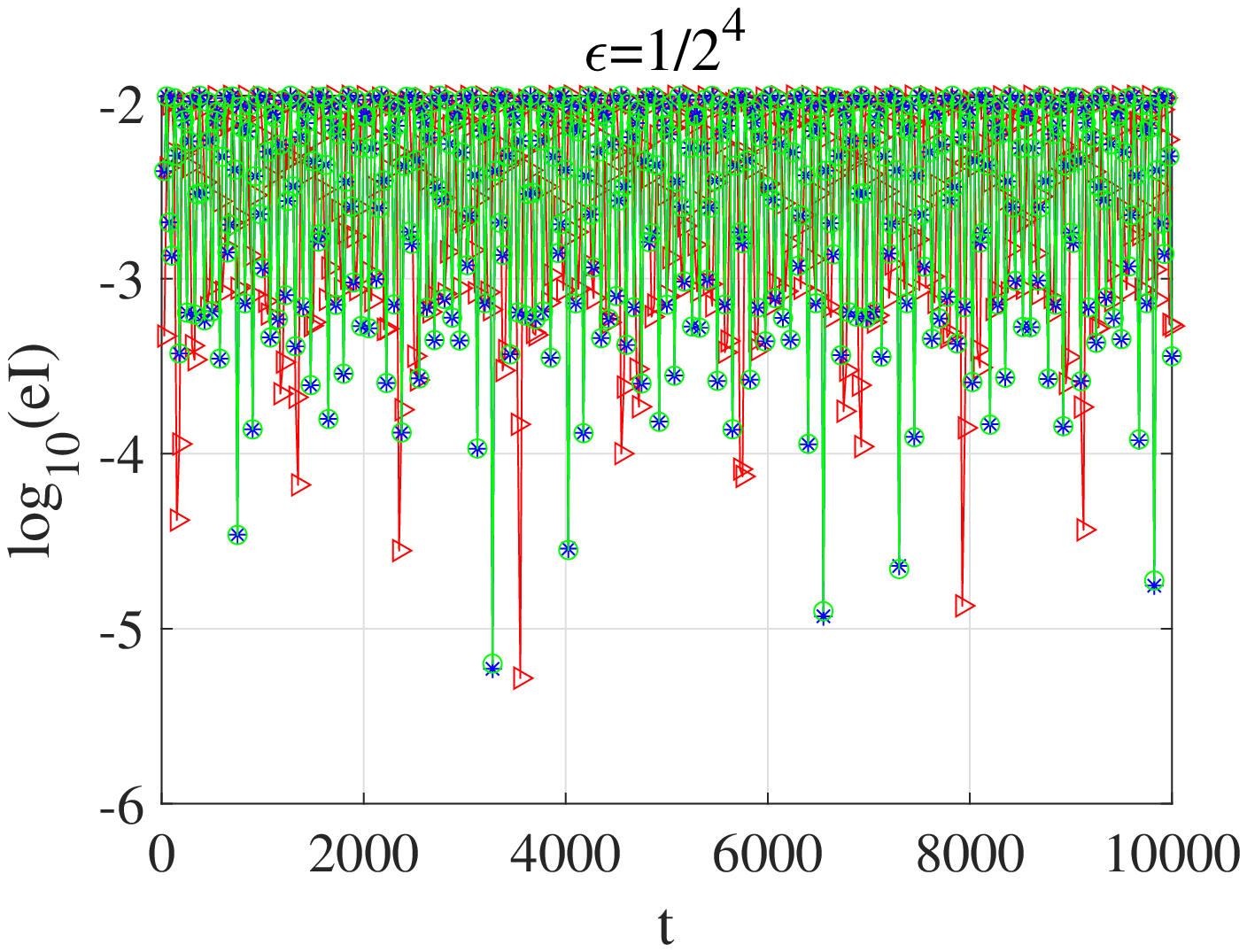}}		
	\end{tabular}
	\caption{{Problem 3.} Evolution of the energy error $e_I$ as function of time $t$. }\label{fig-eI1}
\end{figure}
\begin{figure}[t!]
	\centering
	\begin{tabular}[c]{ccc}%
		\subfigure{\includegraphics[width=4.7cm,height=4.2cm]{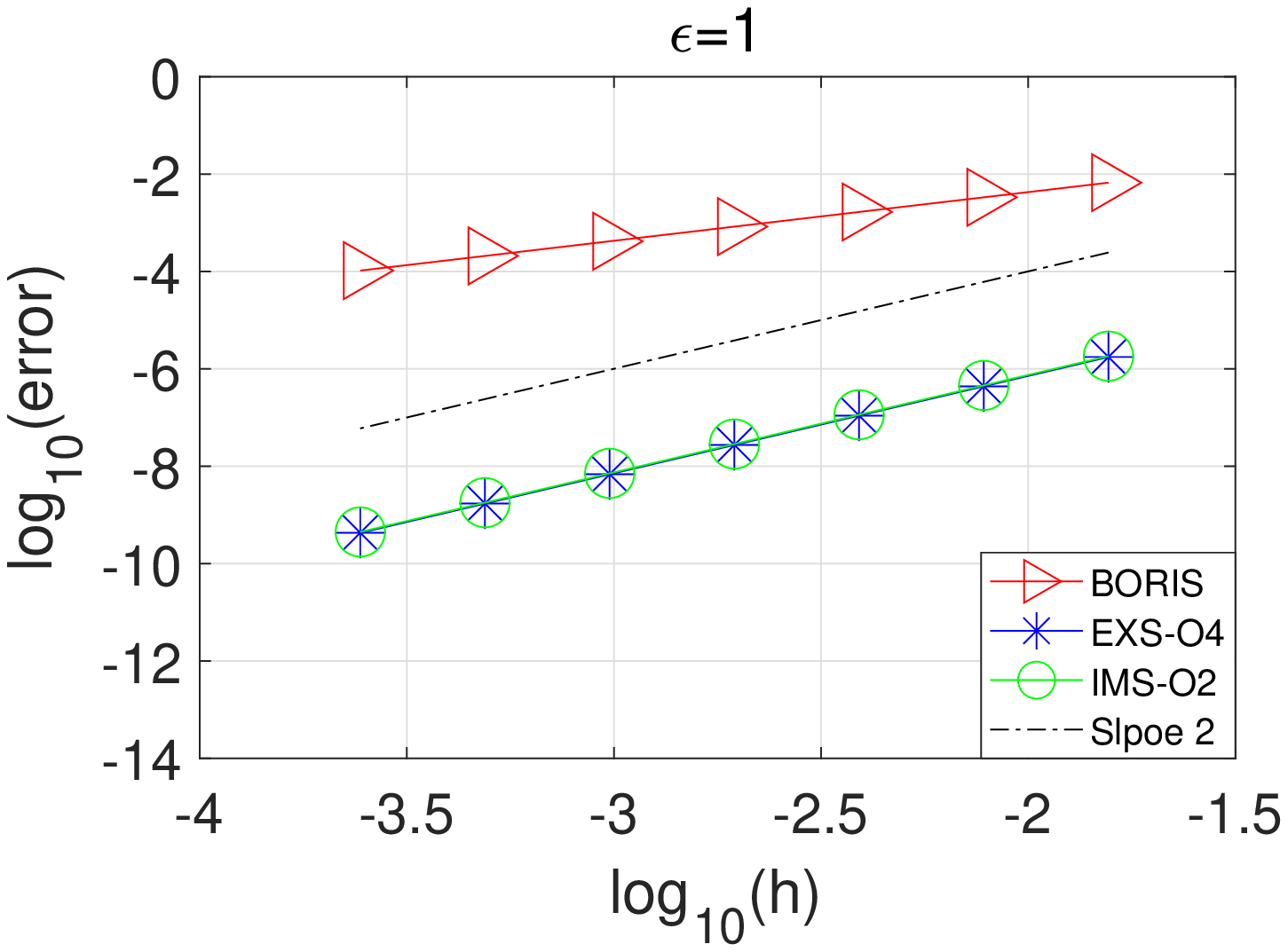}}			\subfigure{\includegraphics[width=4.7cm,height=4.2cm]{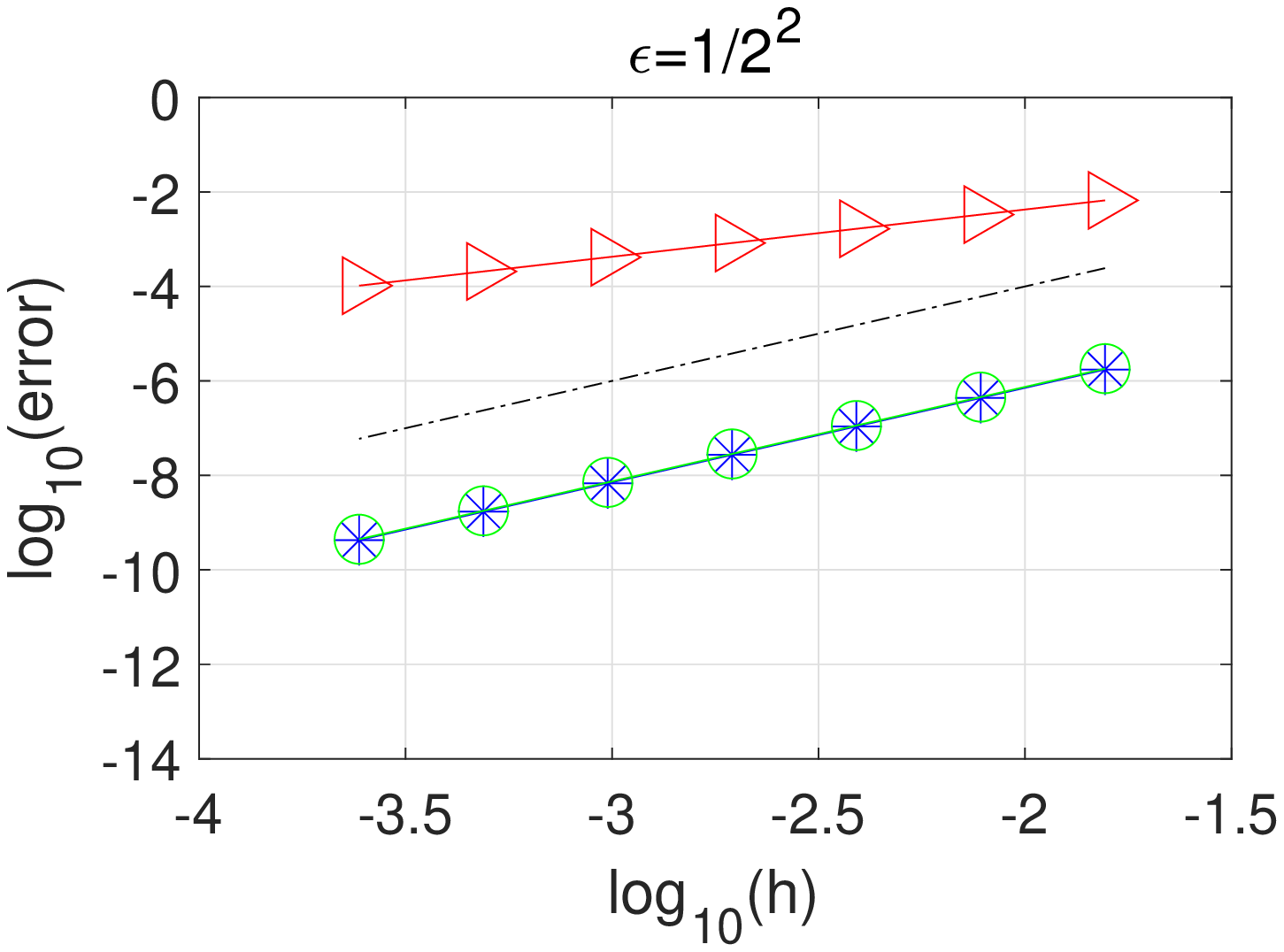}}
		\subfigure{\includegraphics[width=4.7cm,height=4.2cm]{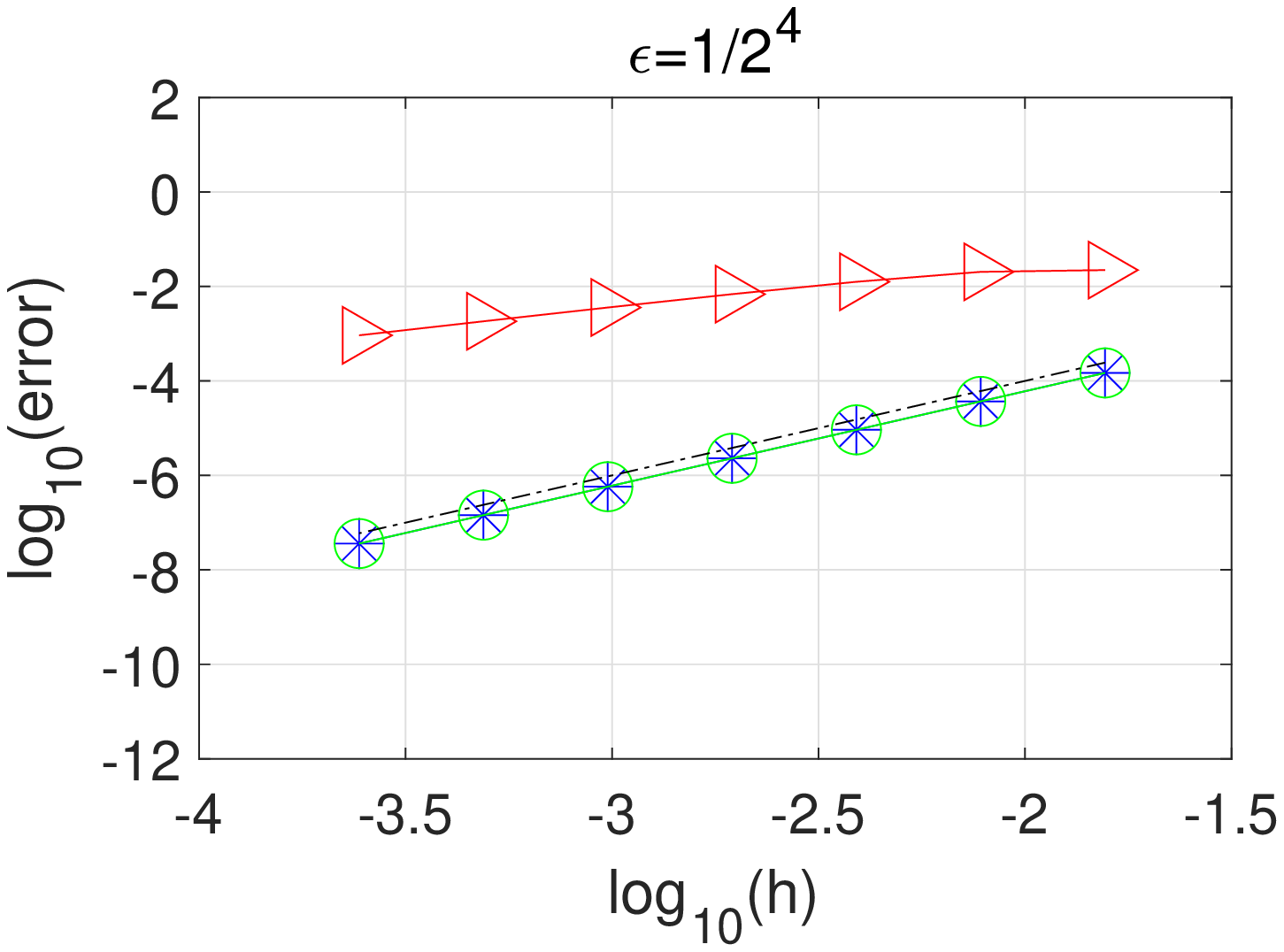}}		
	\end{tabular}
	\caption{{Problem 3.} The global errors \emph{error} with $t = 1$ and $h = \frac{1}{2^k}$ for $k = 6, \dots, 12$ under different $\eps$. }\label{fig-error}
\end{figure}

	$\bullet$ \textbf{Magnetic Moment conservation}. Observing Figs. \ref{fig-eI3}, \ref{fig-eI2} and \ref{fig-eI1},  all the methods have similar magnetic moment {conservations} over long {times.}

	{$\bullet$ \textbf{Accuracy}.  Fig. \ref{fig-error} confirms that our IMS-O2 and EXS-O2   both have second-order accuracy and they perform better than the  Boris algorithm.

The numerical results of Problems 1 and 2 support the theoretical conclusions given in Section \ref{mr}. It is noted that for the general system (Problem \ref{prob3}), the methods also show a long time conservation in the energy, momentum, and magnetic moment, which is more than expected.
	

	\section{Proofs of the main results}\label{proof}
In this section, we    rigorously prove the main results given in Section \ref{mr}. The technical tool {so called  the backward error 	analysis \cite{2006Geometric} which is indispensible for the study of  equations or} numerical solutions over long times will be used in the following proofs.
		\subsection{Backward error analysis}\label{back}
	{Following the analysis of backward error
		analysis, the main idea is to find a modified differential equation whose  solution $z(t)$  at $t=n h$ is equivalent to
the numerical result $x^n$ produced by the considered method.}
	\vskip2mm
	\textbf{EXS-O2.}
	{On the basis of} the first equation in \eqref{eq-EXS} and the symmetry of EXS-O2, we get
	\begin{equation*}
		x^{n-1}=x^n-he^{-\frac{h}{2}\tilde{B} (x^n)}v^n+\frac{h^2}{2}E(x^n).
	\end{equation*}
{This result and the first one of \eqref{eq-EXS} give}
\begin{equation*}\begin{aligned}
	&x^{n+1}-2x^n+x^{n-1}=h\left(e^{\frac{h}{2}\tilde{B} (x^n)}-e^{-\frac{h}{2}\tilde{B} (x^n)}\right)v^n+h^2E(x^n),\\ &x^{n+1}-x^{n-1}=h\left(e^{\frac{h}{2}\tilde{B} (x^n)}+e^{-\frac{h}{2}\tilde{B} (x^n)}\right)v^n.
\end{aligned}
\end{equation*}
 {Elimination of    $v^n$  leads to}
	\begin{equation*} \frac{x^{n+1}-2x^n+x^{n-1}}{h^2}=\frac{2}{h}\frac{e^{\frac{h}{2}\tilde{B}(x^n)}-e^{-\frac{h}{2}\tilde{B}(x^n)}}{e^{\frac{h}{2}\tilde{B}(x^n)}+e^{-\frac{h}{2}\tilde{B}(x^n)}}
		\frac{x^{n+1}-x^{n-1}}{2h}+E(x^n).
	\end{equation*}
	According to  {the special scheme of} $\tilde{B}(x)$, it is obtained that
	\begin{equation*}
		\frac{2}{h}\frac{e^{\frac{h}{2}\tilde{B}(x^n)}-e^{-\frac{h}{2}\tilde{B}(x^n)}}{e^{\frac{h}{2}\tilde{B}(x^n)}+e^{-\frac{h}{2}\tilde{B}(x^n)}}=
		\frac{\tan\left(\frac{h}{2}\abs{B(x^n)}\right)}{\frac{h}{2}\abs{B(x^n)}}\tilde{B}(x^n),
	\end{equation*}
{and such that}
	\begin{equation*}
		\frac{x^{n+1}-2x^n+x^{n-1}}{h^2}
		=\frac{\tan\left(\frac{h}{2}\abs{B(x^n)}\right)}{\frac{h}{2}\abs{B(x^n)}}
		\frac{x^{n+1}-x^{n-1}}{2h}\times B(x^n)+E(x^n).
	\end{equation*}
	For a fixed $t$,  {the function $z(t)$} has to satisfy
	$$\frac{z(t+h)-2z(t)+z(t-h)}{h^2}=\frac{\tan\left(\frac{h}{2}\abs{B(z(t))}\right)}{\frac{h}{2}\abs{B(z(t))}}
	\frac{z(t+h)-z(t-h)}{2h}\times B(z(t))+E(z(t)).$$
	We  {let} $z:=z(t)$ and expand the above functions in powers of $h$, so that
	\begin{equation}\label{modi-EXS}
		\ddot{z}+\frac{h^2}{12}\ddddot{z}+\cdots
		=\frac{\tan\left(\frac{h}{2}\abs{B(z)}\right)}{\frac{h}{2}\abs{B(z)}}(\dot{z}+\frac{h^2}{6}\dddot{z}+\cdots)\times B(z)+E(z).
	\end{equation}

	\textbf{IMS-O2.} {In an analogous way,  the scheme  \eqref{eq-IMS} of IMS-O2 can be formulated as}
	\begin{equation*}
		\begin{aligned}
			&\frac{x^{n+1}-2x^n+x^{n-1}}{h^2}\\
			=&\frac{2}{h}\frac{e^{\frac{h}{2}\tilde{B}(x^n)}-e^{-\frac{h}{2}\tilde{B}(x^n)}}{e^{\frac{h}{2}\tilde{B}(x^n)}+e^{-\frac{h}{2}\tilde{B}(x^n)}}\tilde{B}(x^n)
			\left[\frac{x^{n+1}-x^{n-1}}{2h}-\frac{h}{4}\int_{0}^{1}\left[E\left(\rho x^n+(1-\rho)x^{n+1} \right)-E\left(\rho x^n+(1-\rho)x^{n-1} \right)\right]d\rho\right]\\
			&+\frac{1}{2}\int_{0}^{1}\left[E\left(\rho x^n+(1-\rho)x^{n+1} \right)+E\left(\rho x^n+(1-\rho)x^{n-1}\right)\right]d\rho\\
=&\frac{\tan\left(\frac{h}{2}\abs{B(x^n)}\right)}{\frac{h}{2}\abs{B(x^n)}}
			\left[\frac{x^{n+1}-x^{n-1}}{2h}-\frac{h}{4}\int_{0}^{1}\left[E\left(x^n+\rho(x^{n+1}-x^n)\right)-E\left(x^n+\rho(x^{n-1}-x^n)\right)\right]d\rho\right]\times B(x^n)\\
			&+\frac{1}{2}\int_{0}^{1}\left[E\left(x^n+\rho(x^{n+1}-x^n)\right)+E\left(x^n+\rho(x^{n-1}-x^n)\right)\right]d\rho\\
=&\frac{\tan\left(\frac{h}{2}\abs{B(x^n)}\right)}{\frac{h}{2}\abs{B(x^n)}}
			\left[\frac{x^{n+1}-x^{n-1}}{2h}-\left(\frac{h}{8}E'(x^n)(x^{n+1}-x^{n-1})+\cdots\right)\right]\times B(x^n)\\
			&+E(x^n)+\frac{1}{4}E'(x^n)(x^{n+1}-2x^n+x^{n-1})+\cdots,
		\end{aligned}
	\end{equation*}
{where  the following fact is used here}
	{\begin{align*}
		&\int_{0}^{1}  {E}\left(\rho x^n+(1-\rho)x^{n+1}\right)d\rho
=\int_{0}^{1}{E}\left(\rho x^{n+1}+(1-\rho)x^n\right)d\rho
=\int_{0}^{1}{E}\left(x^n+\rho(x^{n+1}-x^n) \right)d\rho\\
		=&\int_{0}^{1} \left[E(x^n)+E'(x^n)\rho(x^{n+1}-x^n)+\cdots\right] d\rho
=E(x^n)+\frac{1}{2}E'(x^n)(x^{n+1}-x^n)+\cdots.
	\end{align*}}
	Therefore, {the function $z(t)$ satisfies}
	\begin{align*}
		&\frac{z(t+h)-2z(t)+z(t-h)}{h^2}\\=&\frac{\tan\left(\frac{h}{2}\abs{B(z(t))}\right)}{\frac{h}{2}\abs{B(z(t))}}
		\left[\frac{z(t+h)-z(t-h)}{2h}-\left(\frac{h}{8}E'(z(t))(z(t+h)-z(t-h))+\cdots\right)\right]\times B(z(t))\\
		&+E(z(t))+\frac{1}{4}E'(z(t))\left[z(t+h)-2z(t)+z(t-h)\right]+\cdots.
	\end{align*}
{Letting $z:=z(t)$ and expanding the above functions in  powers of $h$, we obtain}
	\begin{equation}\label{modi-IMS}
		\begin{aligned}
			\ddot{z}+\frac{h^2}{12}\ddddot{z}+\cdots
			=&\frac{\tan\left(\frac{h}{2}\abs{B(z)}\right)}{\frac{h}{2}\abs{B(z)}}\left[\left(\dot{z}+\frac{h^2}{6}\dddot{z}+\cdots\right)-\left(\frac{h^2}{4}E'(z)\left(\dot{z}+\frac{h^2}{6}\dddot{z}+\cdots\right)+\cdots\right)\right]\times B(z)\\
			&+E(z)+\frac{h^2}{4}E'(z)\left(\ddot{z}+\frac{h^2}{12}\ddddot{z}+\cdots\right)+\cdots.
		\end{aligned}
	\end{equation}
	
	{We note that when $h=0$, the equations \eqref{modi-EXS} and \eqref{modi-IMS} are both  identical  to \eqref{CPD}}. In fact,  {differentiating these two equations recursively and considering $h=0$ shows that} the third and higher {derivatives depend} on $(z,\dot{z})$. Then we get a modified equation {in even powers of $h$:
	$$\ddot{z}=\frac{\tan\left(\frac{h}{2}\abs{B(z)}\right)}{\frac{h}{2}\abs{B(z)}}\dot{z}\times B(z)+E(z)+h^2F_2(z,\dot{z})+h^4F_4(z,\dot{z})+\cdots,$$
where the coefficient functions  depend on $(z,\dot{z})$.}
	 By   taking inner product {on both sides of  \eqref{modi-EXS} and \eqref{modi-IMS} with different expressions, we can prove Theorem \ref{energy}--\ref{mag-mom} in the following three subsections, respectively}.

	{\subsection{Proof of the energy conservation (Theorem \ref{energy})}\label{modi-ene}}
{Since it has been shown in \cite{22Energy} that	IMS-O2 is an energy-preserving method, here  we just prove the exact conservation of modified energy and {the long} time behaviour of energy} for EXS-O2. {To prove the result, we need to derive two almost invariants which are close to the energy $H(x,v)$. The first one is deduced from the following lemma and the second  is obtained with the help of Lemma \ref{U(x)-ene-EXS}.}
\vskip2mm
	\begin{lem}\label{B(x)-ene-EXS}
		{It is obtained a function $$H_h(x,v)=H(x,v)+h^2H_2(x,v)+h^4H_4(x,v)+\cdots$$ such that}
		\begin{equation*} \frac{d}{dt}H_h{(z,\dot{z})}=\frac{\tan\left(\frac{h}{2}\abs{B(z)}\right)}{\frac{h}{2}\abs{B(z)}}\dot{z}^\intercal\left(\frac{h^2}{3!}\dddot{z}+\frac{h^4}{5!}z^{(5)}+\cdots\right)\times B(z)+\mathcal{O}(h^N)
		\end{equation*}
		along solutions of the modified differential equation \eqref{modi-EXS}, {where the functions $H_{2j}(x,v)$ are independent of the step size $h$ and $\mathcal{O}(h^N)$ is the trunction term.} 
	\end{lem}
	\begin{proof}
	{Multiplying \eqref{modi-EXS} with $\dot{z}^\intercal$ and using the fact that}
		$$\dot{z}^\intercal z^{(2k)}=\frac{d}{dt}\left(\dot{z}^\intercal z^{(2k-1)}-\ddot{z}^\intercal z^{(2k-2)}+\cdots+\frac{(-1)^{k+1}}{2}\left(z^{(k)}\right)^\intercal z^{(k)}\right),$$
	{we get} $$\frac{d}{dt}\left(\frac{1}{2}\dot{z}^\intercal\dot{z}+U(z)
+\frac{h^2}{12}\left(\dot{z}^\intercal\dddot{z}-\frac{1}{2}\ddot{z}^\intercal\ddot{z}\right)+\cdots\right)	=\frac{\tan\left(\frac{h}{2}\abs{B(z)}\right)}{\frac{h}{2}\abs{B(z)}}\dot{z}^\intercal\left(\frac{h^2}{3!}\dddot{z}+\frac{h^4}{5!}z^{(5)}+\cdots\right)\times B(z)+\mathcal{O}(h^N).$$
{Then the result of this lemma is immediately obtained.}
\hfill $ \blacksquare$\end{proof}
	\begin{cor}\label{B-ene-EXS}
{\textbf{(The first almost invariant close to energy)}}
		If the {magnetic field is constant, i.e., $B(x) \equiv B$}, {we have a function
		$$\tilde{H}_h(x,v)=H(x,v)+h^2\tilde{H}_2(x,v)+h^4\tilde{H}_4(x,v)+\cdots$$satisfing}
		\begin{equation*}
			\frac{d}{dt} \tilde{H}_h{(z,\dot{z})} =\mathcal{O}(h^N)
		\end{equation*}
		along solutions of the modified differential equation \eqref{modi-EXS}, {where the functions $\tilde{H}_{2j}(x,v)$ are independent of $h$ and $\mathcal{O}(h^N)$ is the trunction term.}
		
	\end{cor}
	\begin{proof}
		 Based on the results in Lemma \ref{B(x)-ene-EXS}  and the fact that
		$$\dot{z}^\intercal\left(z^{(2k+1)}\times B\right)=\frac{d}{dt}\left(\dot{z}^\intercal\left(z^{(2k)}\times B\right)-\ddot{z}^\intercal\left(z^{(2k-1)}\times B\right)+\cdots+(-1)^{k+1}\left(z^{(k)}\right)^\intercal\left(z^{(k+1)}\times B\right)\right),$$
the first statement is obtained. 
{Then, it is arrived that  $$\frac{d}{dt}\left(\frac{1}{2}\dot{z}^\intercal\dot{z}+U(z)+\frac{h^2}{12}\left(\dot{z}^\intercal\dddot{z}-\frac{1}{2}\ddot{z}^\intercal\ddot{z}\right)-\frac{\tan\left(\frac{h}{2}\abs{B}\right)}{\frac{h}{2}\abs{B}}\left(\frac{h^2}{3!}\dot{z}^\intercal(\ddot{z}\times B)+\cdots\right)+\cdots\right)
		=\mathcal{O}(h^N),$$
which completes the proof.}
	\hfill $ \blacksquare$\end{proof}
{	
\begin{rem}	It is noted that based on this  corollary,    the method EXS-O2 will be shown to have a long-time near-conservation of the energy for a constant magnetic field $B$.
\end{rem}	 }
	\begin{lem}\label{U(x)-ene-EXS}
		{There exists a function
		$$H_h(x,v)=H(x,v)+h^2H_2(x,v)+h^4H_4(x,v)+\cdots$$
		and it satisfies}
		\begin{equation*}
			\frac{d}{dt}H_h{(z,\dot{z})}=\left(\frac{h^2}{3!}\dddot{z}+\frac{h^4}{5!}z^{(5)}+\cdots\right)^\intercal E(z)+\mathcal{O}(h^N)
		\end{equation*}
		along solutions obtained by the modified differential equation \eqref{modi-EXS},  { where the functions $H_{2j}(x,v)$ (different from
			those in Lemma \ref{B(x)-ene-EXS}) don't depend on the step size $h$ and $\mathcal{O}(h^N)$ is the trunction term.}
	\end{lem}
	\begin{proof}
 {It is noted that if $l+m$ is odd, ${z^{(l)}}^\intercal z^{(m)}$ can be written as a total differential. Then}
 taking inner product {on both sides of  \eqref{modi-EXS} with $\left(\dot{z}+\frac{h^2}{6}\dddot{z}+\cdots\right)$ and using the same arguments  of Lemma \ref{B(x)-ene-EXS}, one has} $$\frac{d}{dt}\left(\frac{1}{2}\dot{z}^\intercal\dot{z}+U(z)+\frac{h^2}{12}\left(\dot{z}^\intercal\dddot{z}-\frac{1}{2}\ddot{z}^\intercal\ddot{z}\right)+\cdots\right)
		=\left(\frac{h^2}{3!}\dddot{z}+\frac{h^4}{5!}z^{(5)}+\cdots\right)^\intercal E(z)+\mathcal{O}(h^N).$$
 {The proof is complete.}	\hfill $ \blacksquare$\end{proof}
	
	\begin{cor}\label{Q-ene-EXS}
{\textbf{(The second almost invariant close to energy)}}
		{If  $U(x)=\frac{1}{2}x^\intercal Qx + q^\intercal x$,} {we get a function
		$$\widehat H_h(x,v)=H(x,v)+h^2\widehat H_2(x,v)+h^4\widehat H_4(x,v)+\cdots$$
		 such that}
		\begin{equation*}
			\frac{d}{dt}\widehat{H}_h{(z,\dot{z})}=\mathcal{O}(h^N)
		\end{equation*}
		along solutions of the modified differential equation \eqref{modi-EXS}, {where the functions $H_{2j}(x,v)$ are independent of $h$ and $\mathcal{O}(h^N)$ is the trunction term}.
	\end{cor}
	\begin{proof}
{Since the expression ${z^{(2k+1)}}^\intercal \nabla U(z)={z^{(2k+1)}}^\intercal(Qz+q)$ is a total differential:}
		$$\left({z^{(2k+1)}}\right)^\intercal(Qz+q)=\frac{d}{dt}\left(\left({z^{(2k)}}\right)^\intercal(Qz+q)-\left( {z^{(2k-1)}}\right) ^\intercal Q\dot{z}+\cdots+\frac{(-1)^k}{2}\left({z^{(k)}}\right)^\intercal Qz^{(k)}\right),$$
{it is clear that}	$$\frac{d}{dt}\left(\frac{1}{2}\dot{z}^\intercal\dot{z}+U(z)+\frac{h^2}{12}\left(\dot{z}^\intercal\dddot{z}-\frac{1}{2}\ddot{z}^\intercal\ddot{z}\right)+\frac{h^2}{6}\left(\ddot{z}^\intercal(Qz+q)-\frac{1}{2}\dot{z}^\intercal Q\dot{z}\right)+\cdots\right)
		=\mathcal{O}(h^N).$$
	\hfill $ \blacksquare$\end{proof}

{	
\begin{rem} {This result   confirms that the method EXS-O2 hold a long-term near-conservation of the total energy as long as the scalar potential is quadratic.}
\end{rem}	
Based on the above preparations, we are in the position to prove the results \eqref{eq-energy} and \eqref{Hherr} of Theorem \ref{energy}.}

\vskip2mm
\textbf{Proof of \eqref{eq-energy}.}
{
The result is firstly shown for a constant magnetic field $B$. From Corollary \ref{B-ene-EXS}, it follows that}
\begin{equation}\label{H-h^2}
\begin{aligned}
	H(x^n,v^n)&=\tilde{H}_h(x^n,v^n)+\mathcal{O}(h^2)=\tilde{H}_h(x^0,v^0)+\sum_{k=1}^{n}\left(\tilde{H}_h(x^k,v^k)-\tilde{H}_h(x^{k-1},v^{k-1})\right)+\mathcal{O}(h^2)\\
	&=H(x^0,v^0)+\mathcal{O}(h^2)+n\mathcal{O}(h^{N+1})+\mathcal{O}(h^2)=H(x^0,v^0)+\mathcal{O}(h^2).
\end{aligned}
\end{equation}
The last equation is {meaningful} if and only if $nh^{N+1}\le h^2$, i.e. $nh\le h^{2-N}$ {and this gives the desired bound for the deviation of the total energy along the numerical solution.} {
{Moreover,} for the quadratic scalar potential,  Corollary \ref{Q-ene-EXS} implies
$H(x^n,v^n) =\hat{H}_h(x^n,v^n)+\mathcal{O}(h^2).$
Using the same derivation as stated above,  it is easy to show that the result  \eqref{eq-energy} also holds for the quadratic scalar potential.}
	  \hfill $ \blacksquare$
\vskip2mm
{\textbf{Proof of \eqref{Hherr}.}}  	The second equation in \eqref{eq-EXS} {can be reformulated as
		$$e^{-\frac{h}{2}\tilde{B} (x^{n+1})}v^{n+1}=e^{\frac{h}{2}\tilde{B} (x^n)}v^n+\frac{h}{2}[E(x^n)+E(x^{n+1})].$$
This result and the fact that $\tilde{B}$ is skew-symmetric imply}
		\begin{align*}
			\frac{1}{2}({v^{n+1}})^\intercal v^{n+1}-\frac{1}{2}({v^n})^\intercal v^n
			&=\frac{1}{2}\left(e^{-\frac{h}{2}\tilde{B} (x^{n+1})}{v^{n+1}}\right)^\intercal \left(e^{-\frac{h}{2}\tilde{B} (x^{n+1})}v^{n+1}\right)-\frac{1}{2}\left(e^{\frac{h}{2}\tilde{B} (x^n)}{v^n}\right)^\intercal \left(e^{\frac{h}{2}\tilde{B} (x^n)}v^n\right)\\
			&=\frac{1}{2}\left[E(x^n)+E(x^{n+1})\right]^\intercal\left[he^{\frac{h}{2}\tilde{B} (x^n)}v^n+\frac{h^2}{4}\left[E(x^n)+E(x^{n+1})\right]\right].
		\end{align*}
{Besides, the  deviation of $U$ at $x^{n+1}$ and $x^{n}$ can be expressed as}
		\begin{align*}
			U(x^{n+1})-U(x^n)&=\frac{1}{2}({x^{n+1}})^\intercal Qx^{n+1}+q^\intercal  x^{n+1}-\left(\frac{1}{2}({x^n})^\intercal Qx^n+q^\intercal  x^n\right)\\
			&=\frac{1}{2}({x^{n+1}}+x^n)^\intercal Q(x^{n+1}-x^n)+q^\intercal  (x^{n+1}-x^n)\\
			&=\left[\frac{1}{2}({x^{n+1}}+x^n)^\intercal Q+q^\intercal\right]{\left(he^{\frac{h}{2}\tilde{B} (x^n)}v^n+\frac{h^2}{2}E(x^n)\right)}.
		\end{align*}
	{Therefore} $$\frac{1}{2}({v^{n+1}})^\intercal v^{n+1}-\frac{1}{2}({v^n})^\intercal v^n+U(x^{n+1})-U(x^n)=\frac{h^2}{8}\abs{\nabla U(x^{n+1})}^2-\frac{h^2}{8}\abs{\nabla U(x^n)}^2$$
{and this further gives} $H_h(x^{n+1},v^{n+1})=H_h(x^n,v^n).$
	\hfill $ \blacksquare$
	{\subsection{Proof of the momentum conservation (Theorem \ref{momentum})}}\label{modi-mom}
{The proof is given by  finding   two almost invariants which are close to momentum. To derive these invariants, we first present the following lemma.}
\vskip2mm
	\begin{lem}\label{B(x)-mom}
		{The following two
		functions
		\begin{equation*}
			\begin{aligned}&M_h^e(x,v)=M(x,v)+h^2M_2^e(x,v)+h^4M_4^e(x,v)+\cdots,\\
&M_h^i(x,v)=M(x,v)+h^2M_2^i(x,v)+h^4M_4^i(x,v)+\cdots	,	\end{aligned}
		\end{equation*}
		 can be derived with the $h$-independent  functions $M_{2j}^e(x,v)$ and $M_{2j}^i(x,v)$, and they  satisfy}
		\begin{equation*}
			\frac{d}{dt}M_h^e{(z,\dot{z})}=\frac{\tan\left(\frac{h}{2}\abs{B(z)}\right)}{\frac{h}{2}\abs{B(z)}}z^\intercal S\left(\frac{h^2}{3!}\dddot{z}+\frac{h^4}{5!}z^{(5)}+\cdots\right)\times B(z)+\mathcal{O}(h^N)
		\end{equation*}
		along solutions of the modified differential equation \eqref{modi-EXS} for EXS-O2 and
		\begin{equation*}
			\begin{aligned}
				\frac{d}{dt}M_h^i{(z,\dot{z})}&=\frac{\tan\left(\frac{h}{2}\abs{B(z)}\right)}{\frac{h}{2}\abs{B(z)}}z^\intercal S\left[\left(\frac{h^2}{3!}\dddot{z}+{\frac{h^4}{5!}z^{(5)}}+\cdots\right)-\frac{h^2}{4}E'(z)\left(\dot{z}+\frac{h^2}{3!}\dddot{z}+\cdots\right)+\cdots\right]\times B(y)\\
				&+\frac{h^2}{4}z^\intercal SE'(z)\left(\ddot{z}+\frac{h^2}{12}\ddddot{z}+\cdots\right)+\cdots+\mathcal{O}(h^N)
			\end{aligned}
		\end{equation*}
		along solutions of the modified differential equation \eqref{modi-IMS} for IMS-O2.
	\end{lem}
	\begin{proof}{We first multiply \eqref{modi-EXS} and \eqref{modi-IMS} with $\dot{z}^\intercal S$.  Then notice the fact that  {the expression $z^\intercal Sz^{(2k)}$ takes a form of total derivative:}
		$$z^\intercal Sz^{(2k)}=\frac{d}{dt}\left(z^\intercal Sz^{(2k-1)}-\dot{z}^\intercal Sz^{(2k-2)}+\cdots+(-1)^{k-1}\left({z^{(k-1)}}\right)^\intercal Sz^{(k)}\right).$$
		In addition, the invariance properties \eqref{in-prop} show that $z^\intercal S\nabla U(z)=0$ and $z^\intercal S(\dot{z}\times B(z))=-\frac{d}{dt}(z^\intercal SA(z))$  \cite{18Energy}. Based on the above results, we get}
		\begin{align*}
			&\frac{d}{dt}\left(z^\intercal S\dot{z}+\frac{\tan\left(\frac{h}{2}\abs{B(z)}\right)}{\frac{h}{2}\abs{B(z)}}z^\intercal SA(z)+\frac{h^2}{12}(z^\intercal S\dddot{z}-\dot{z}^\intercal S\ddot{z})+\cdots\right)\\
			&=\frac{\tan\left(\frac{h}{2}\abs{B(z)}\right)}{\frac{h}{2}\abs{B(z)}}z^\intercal S\left(\frac{h^2}{3!}\dddot{z}+\frac{h^4}{5!}z^{(5)}+\cdots\right)\times B(z)+\mathcal{O}(h^N)
		\end{align*}
		and by expansion of $\frac{\tan x}{x}$, the above equation can be rewritten as
		\begin{align*}
			&\frac{d}{dt}\left(z^\intercal S\dot{z}+z^\intercal SA(z)+\frac{h^2}{3}\left(\frac{\abs{B(z)}}{2}\right)^2z^\intercal SA(z)+\frac{h^2}{12}(z^\intercal S\dddot{z}-\dot{z}^\intercal S\ddot{z})+\cdots\right)\\
			&=\frac{\tan\left(\frac{h}{2}\abs{B(z)}\right)}{\frac{h}{2}\abs{B(z)}}z^\intercal S\left(\frac{h^2}{3!}\dddot{z}+\frac{h^4}{5!}z^{(5)}+\cdots\right)\times B(z)+\mathcal{O}(h^N)
		\end{align*}
		for EXS-O2.
		
		Similarly, we can  get
		\begin{align*}
			&\frac{d}{dt}\left(z^\intercal S\dot{z}+z^\intercal SA(z)+\frac{h^2}{3}\left(\frac{\abs{B(z)}}{2}\right)^2z^\intercal SA(z)+\frac{h^2}{12}(z^\intercal S\dddot{z}-\dot{z}^\intercal S\ddot{z})+\cdots\right)\\
=&\frac{\tan\left(\frac{h}{2}\abs{B(z)}\right)}{\frac{h}{2}\abs{B(z)}}z^\intercal S\left[\left(\frac{h^2}{3!}\dddot{z}+{\frac{h^4}{5!}z^{(5)}}+\cdots\right)-\frac{h^2}{4}E'(z)\left(\dot{z}+\frac{h^2}{3!}\dddot{z}+\cdots\right)+\cdots\right]\times B(z)\\
			&+\frac{h^2}{4}z^\intercal SE'(z)\left(\ddot{z}+\frac{h^2}{12}\ddddot{z}+\cdots\right)+\cdots+\mathcal{O}(h^N)
		\end{align*}
		for IMS-O2.	
	\hfill $ \blacksquare$\end{proof}

	{
  The above results can be improved under some conditions, which is stated by the following corollary. }
	\begin{cor}\label{B-mom} {\textbf{(The first almost invariant close to momentum)}}
		{If $B(x) \equiv B$} and $x^\intercal\nabla U(x)\times B=0$ for all $x$, {it is obtained that }the
		{functions}
		$$\widehat{M}_h^e(x,v)=M(x,v)+h^2\widehat{M}_2^e(x,v)+h^4\widehat{M}_4^e(x,v)+\cdots,$$
		$$\widehat{M}_h^i(x,v)=M(x,v)+h^2\widehat{M}_2^i(x,v)+h^4\widehat{M}_4^i(x,v)+\cdots,$$
  satisfy
		\begin{equation*}
			\frac{d}{dt}\widehat{M}_h^e{(z,\dot{z})}=\mathcal{O}(h^N)
		\end{equation*}
		along solutions of the modified differential equation \eqref{modi-EXS} for EXS-O2 and
		\begin{equation*}
			\begin{aligned}
				\frac{d}{dt}\widehat{M}_h^i{(z,\dot{z})}=&\frac{\tan\left(\frac{h}{2}\abs{B(z)}\right)}{\frac{h}{2}\abs{B(z)}}z^\intercal S\left[-\frac{h^2}{4}E'(z)\left(\dot{z}+\frac{h^2}{6}\dddot{z}+\cdots\right)+\cdots\right]\times B(z)\\
				&+\frac{h^2}{4}z^\intercal SE'(z)\left(\ddot{z}+\frac{h^2}{12}\ddddot{z}+\cdots\right)+\cdots+\mathcal{O}(h^N)
			\end{aligned}
		\end{equation*}
		along solutions of the modified differential equation \eqref{modi-IMS} for IMS-O2. {Here the functions $\widehat{M}_{2j}^e(x,v)$ and $\widehat{M}_{2j}^i(x,v)$ are $h$-independent and  $\mathcal{O}(h^N)$ is the  truncation term.}
	\end{cor}
	\begin{proof}
		{Using the same way as that of Lemma \ref{B(x)-mom} and noticing that  $z^\intercal S(z^{(2k+1)}\times B)$ is a total derivative (due to the symmetry of $S^2$):}
		$$z^\intercal S\left(z^{(2k+1)}\times B\right)=z^\intercal S^2z^{(2k+1)}=\frac{d}{dt}\left(z^\intercal S^2z^{(2k)}-\dot{z}^\intercal S^2z^{(2k-1)}+\cdots+\frac{(-1)^k}{2}\left({z^{(k)}}\right)^\intercal S^2z^{(k)}\right),$$
{we get}
		\begin{align*}
			&\frac{d}{dt}\left(z^\intercal S\dot{z}+z^\intercal SA(z)+\frac{h^2}{3}\left(\frac{\abs{B}}{2}\right)^2z^\intercal SA(z)+\frac{h^2}{12}(z^\intercal S\dddot{z}-\dot{z}^\intercal S\ddot{z})-\frac{h^2}{6}\frac{\tan\left(\frac{h}{2}\abs{B}\right)}{\frac{h}{2}\abs{B}}\left(z^\intercal S^2\ddot{z}{-\frac{1}{2}\dot{z}^\intercal S^2\dot{z}}\right)+\cdots\right)\\&=\mathcal{O}(h^N)
		\end{align*}
		for EXS-O2 and
		\begin{equation}\label{eq-B-mom-IMS}
			\begin{aligned}
				&\frac{d}{dt}\left(z^\intercal S\dot{z}+z^\intercal SA(z)+\frac{h^2}{3}\left(\frac{\abs{B}}{2}\right)^2z^\intercal SA(z)+\frac{h^2}{12}(z^\intercal S\dddot{z}-\dot{z}^\intercal S\ddot{z})-\frac{h^2}{6}\frac{\tan\left(\frac{h}{2}\abs{B}\right)}{\frac{h}{2}\abs{B}}\left(z^\intercal S^2\ddot{z}{-\frac{1}{2}\dot{z}^\intercal S^2\dot{z}}\right)+\cdots\right)\\
=&\frac{\tan\left(\frac{h}{2}\abs{B}\right)}{\frac{h}{2}\abs{B}}z^\intercal S\left[-\frac{h^2}{4}E'(z)\left(\dot{z}+\frac{h^2}{6}\dddot{z}+\cdots\right)+\cdots\right]\times B+\frac{h^2}{4}z^\intercal SE'(z)\left(\ddot{z}+\frac{h^2}{12}\ddddot{z}+\cdots\right)+\cdots+\mathcal{O}(h^N)
			\end{aligned}
		\end{equation}
		for IMS-O2.
	\hfill $ \blacksquare$\end{proof}
	\begin{cor}
{\textbf{(The second almost invariant close to momentum)}}
		{Suppose that   the conditions in Corollary \ref{B-mom} are satisfied and} {$U(x)=\frac{1}{2}x^\intercal Qx + q^\intercal x$ with $QS=SQ$.} {Then} the function
		$$\tilde{M}_h^i(x,v)=M(x,v)+h^2\tilde{M}_2^i(x,v)+h^4\tilde{M}_4^i(x,v)+\cdots$$
		{satisfies}
		\begin{equation*}
			\frac{d}{dt}\tilde{M}_h^i{(z,\dot{z})}=\mathcal{O}(h^N)
		\end{equation*}
		along solutions of the modified differential equation \eqref{modi-IMS} for IMS-O2. {Here the functions $\tilde{M}_{2j}^i(x,v)$ {don't depend on }the step size $h$ and $\mathcal{O}(h^N)$ stands for the truncation term.}
	\end{cor}
	\begin{proof}
		{Based on the assumptions}, the right-hand side of \eqref{eq-B-mom-IMS} can be simplified as
		\begin{align*}
			&\frac{\tan\left(\frac{h}{2}\abs{B}\right)}{\frac{h}{2}\abs{B}}z^\intercal S\left[-\frac{h^2}{4}E'(z)\left(\dot{z}+\frac{h^2}{6}\dddot{z}+\cdots\right)+\cdots\right]\times B+\frac{h^2}{4}z^\intercal SE'(z)\left(\ddot{z}+\frac{h^2}{12}\ddddot{z}+\cdots\right)+\cdots+\mathcal{O}(h^N)\\
			&=\frac{\tan\left(\frac{h}{2}\abs{B}\right)}{\frac{h}{2}\abs{B}}z^\intercal S\left[\frac{h^2}{4}Q\left(\dot{z}+\frac{h^2}{6}\dddot{z}+\cdots\right)\right]\times B-\frac{h^2}{4}z^\intercal SQ\left(\ddot{z}+\frac{h^2}{12}\ddddot{z}+\cdots\right)+{\mathcal{O}(h^N).}
		\end{align*}According to the properties of $S$ and $Q$, it is obtained that $SQ$ is skew-symmetric and $S^2Q$ is symmetric. Then, we get
		$$z^\intercal S\left(\left(Qz^{(2k+1)}\right)\times B\right)=z^\intercal S^2Qz^{(2k+1)}=\frac{d}{dt}\left(z^\intercal S^2Qz^{(2k)}-\dot{z}^\intercal S^2Qz^{(2k-1)}+\cdots+\frac{(-1)^k}{2}\left({z^{(k)}}\right)^\intercal S^2Qz^{(k)}\right)$$
		and
		$$z^\intercal SQz^{(2k)}=\frac{d}{dt}\left(z^\intercal SQz^{(2k-1)}-\dot{z}^\intercal SQz^{(2k-2)}+\cdots+(-1)^{k-1}\left({z^{(k-1)}}\right)^\intercal SQz^{(k)}\right).$$
{Therefore, the proof is complete.}
	\hfill $ \blacksquare$\end{proof}

 \textbf{Proof of \eqref{eq-momentum}.} {With the results stated above,  the method EXS-O2 conserves the momentum with the accuracy}
	\begin{equation}\label{M-h^2}
		\begin{aligned}
			M(x^n,v^n)&=\tilde{M}_h^e(x^n,v^n)+\mathcal{O}(h^2)=\tilde{M}_h^e(x^0,v^0)+\sum_{k=1}^{n}\left(\tilde{M}_h^e(x^k,v^k)-\tilde{M}_h^e(x^{k-1},v^{k-1})\right)+\mathcal{O}(h^2)\\
			&=M(x^0,v^0)+\mathcal{O}(h^2)+n\mathcal{O}(h^{N+1})+\mathcal{O}(h^2)=M(x^0,v^0)+\mathcal{O}(h^2),
		\end{aligned}
	\end{equation}
	{as long as  $nh\le h^{2-N}$. For IMS-O2, the same discussion applies to $M_h^i(x,v)$ } {and then \eqref{eq-momentum} can be proved.}

 \hfill $ \blacksquare$
	\subsection{Proof of the magnetic moment conservation (Theorem \ref{mag-mom})}\label{modi-mag mom}
	In this section, we just considered magnetic moment in constant magnetic field $B$, that is
	$$I(x,v)=\frac{\abs{v\times B}^2}{2\abs{B}^3}=\frac{\abs{v\times b}^2}{2\abs{B}}$$
	with $b=\frac{B}{\abs{B}}$ and a skew-symmetric matrix $\widehat{B}=\frac{\tilde{B}}{\abs{B}}$ 	{satisfying} $v\times b=\widehat{B}v$.
	
	{Specifically,} the modified equations \eqref{modi-EXS} and \eqref{modi-IMS} are respectively} equivalent to
	\begin{equation}\label{modi-EXS-2}
		\left(\dot{z}+\frac{h^2}{6}\dddot{z}+\cdots\right)\times b
		=\frac{h}{2\tan\left(\frac{h}{2}\abs{B}\right)}\left[\left(\ddot{z}+\frac{h^2}{12}\ddddot{z}+\cdots\right)-E(z)\right]
	\end{equation}
	and
	\begin{equation}\label{modi-IMS-2}
		\begin{aligned}
			\left(\dot{z}+\frac{h^2}{6}\dddot{z}+\cdots\right)\times b =&\frac{h}{2\tan\left(\frac{h}{2}\abs{B}\right)}\left[\left(\ddot{z}+\frac{h^2}{12}\ddddot{z}+\cdots\right)-\left(E(z)+\frac{h^2}{4}E'(z)\left(\ddot{z}+\frac{h^2}{12}\ddddot{z}+\cdots\right)+\cdots\right)\right]\\
			&+\left(\frac{h^2}{4}E'(z)\left(\dot{z}+\frac{h^2}{6}\dddot{z}+\cdots\right)+\cdots\right)\times b.
		\end{aligned}
	\end{equation}
{Based on these results, we can derive two almost invariants   close to magnetic moment, which will complete the proof of Theorem \ref{mag-mom}. To this end,  we first prove the following result.}

	\begin{lem}\label{B-mag-mom}
		If {$B(x)\equiv B$}, {one gets that  the  functions
		$$I_h^e(x,v)=I(x,v)+h^2I_2^e(x,v)+h^4I_4^e(x,v)+\cdots,$$
		$$I_h^i(x,v)=I(x,v)+h^2I_2^i(x,v)+h^4I_4^i(x,v)+\cdots,$$
satisfy}
		\begin{equation*}
			\frac{d}{dt}I_h^e{(z,\dot{z})}=-\frac{h}{2\abs{B}\tan\left(\frac{h}{2}\abs{B}\right)}(\ddot{z}\times b)^\intercal E(z)+\mathcal{O}(h^N)
		\end{equation*}
		along solutions of the modified differential equation \eqref{modi-EXS} for EXS-O2 and
		\begin{equation*}
			\begin{aligned}
				\frac{d}{dt}I_h^i{(z,\dot{z})}
=&-\frac{h}{2 \abs{B}\tan\left(\frac{h}{2}\abs{B}\right)}(\ddot{z}\times b)^\intercal\left[E(z)+\frac{h^2}{4}E'(z)\left(\ddot{z}+\frac{h^2}{12}\ddddot{z}+\cdots\right)+\cdots\right]\\
				&+\frac{1}{\abs{B}}(\ddot{z}\times b)^\intercal \left[\frac{h^2}{4}E'(z)\left(\dot{z}+\frac{h^2}{6}\dddot{z}+\cdots\right)+\cdots\right]\times b+\mathcal{O}(h^N)
			\end{aligned}
		\end{equation*}
		along solutions of the modified differential equation \eqref{modi-IMS} for IMS-O2. {Here the functions  $I_{2j}^e(x,v)$ and $I_{2j}^i(x,v)$  are $h$-independent and $\mathcal{O}(h^N)$ is the truncation term.}
	\end{lem}
	\begin{proof}
		{Multiply} \eqref{modi-EXS-2} and \eqref{modi-IMS-2} with $\frac{1}{\abs{B}}(\ddot{z}\times b)^\intercal$.   {It is clear that}
		\begin{align*}
			&\frac{1}{\abs{B}}(\ddot{z}\times b)^\intercal \left(z^{(2k+1)}\times b\right)\\
			&=\begin{cases}
				\frac{d}{dt}I(z,\dot{z}), & \text{$k=0$}\\
				\frac{d}{dt}\frac{1}{\abs{B}}\left((\ddot{z}\times b)^\intercal\left(z^{(2k)}\times b\right)-\left(\dddot{z}\times b\right)^\intercal\left(z^{(2k-1)}\times b\right)+\cdots+\frac{(-1)^{k+1}}{2}\left(z^{(k+1)}\times b\right)^\intercal\left(z^{(k+1)}\times b\right)\right), & \text{$k\in \mathbb{N}^*$}
			\end{cases}
		\end{align*}
		and
		\begin{align*}
			&\frac{1}{\abs{B}}(\ddot{z}\times b)^\intercal z^{(2k)}\\
			&=\begin{cases}
				0, & \text{$k=1$}\\
				\frac{d}{dt}\frac{1}{\abs{B}}\left((\ddot{z}\times b)^\intercal z^{(2k-1)}-\left(\dddot{z}\times b\right)^\intercal z^{(2k-2)}+\cdots+(-1)^k\left(z^{(k)}\times b\right)^\intercal z^{(k+1)}\right), & \text{$k\in \mathbb{N}^*-\{1\}$}
			\end{cases}
		\end{align*}
{These immediately} demonstrate that
		\begin{equation*}\label{B-mag-mom-EXS}
			\begin{aligned}
				&\frac{d}{dt}\left[I(z,\dot{z})+\frac{h^2}{12\abs{B}}(\ddot{z}\times b)^\intercal (\ddot{z}\times b)-\frac{h}{2\abs{B}\tan\left(\frac{h}{2}\abs{B}\right)}\left(\frac{h^2}{12}(\ddot{z}\times b)^\intercal\dddot{z}+\cdots\right)\right]\\
				&=-\frac{h}{2\abs{B}\tan\left(\frac{h}{2}\abs{B}\right)}(\ddot{z}\times b)^\intercal E(z)+\mathcal{O}(h^N)
			\end{aligned}
		\end{equation*}
		for EXS-O2.
		
		On the other hand {and with the same arguments, we   get}
		\begin{equation}\label{B-mag-mom-IMS}
			\begin{aligned}
				&\frac{d}{dt}\left[I(z,\dot{z})+\frac{h^2}{12\abs{B}}(\ddot{z}\times b)^\intercal (\ddot{z}\times b)-\frac{h}{2\abs{B}\tan\left(\frac{h}{2}\abs{B}\right)}\left(\frac{h^2}{12}(\ddot{z}\times b)^\intercal\dddot{z}+\cdots\right)\right]\\
=&-\frac{h}{2\abs{B} \tan\left(\frac{h}{2}\abs{B}\right)}(\ddot{z}\times b)^\intercal\left[E(z)+\frac{h^2}{4}E'(z)\left(\ddot{z}+\frac{h^2}{12}\ddddot{z}+\cdots\right)+\cdots\right]\\
				&+\frac{1}{\abs{B}}(\ddot{z}\times b)^\intercal \left[\frac{h^2}{4}E'(z)\left(\dot{z}+\frac{h^2}{6}\dddot{z}+\cdots\right)+\cdots\right]\times b+\mathcal{O}(h^N)
			\end{aligned}
		\end{equation}
		for IMS-O2.	
	\hfill $ \blacksquare$\end{proof}
	\begin{cor} {\textbf{(The  almost invariants close to magnetic moment)}}
		{If $B(x)\equiv B$ and $U(x)=\frac{1}{2}x^\intercal Qx + q^\intercal x$ with} $Q\widehat{B}=\widehat{B}Q${, one obtains that} the {functions}
		$$\widehat{I}_h^e(x,v)=I(x,v)+h^2\widehat{I}_2^e(x,v)+h^4\widehat{I}_4^e(x,v)+\cdots,$$
		$$\widehat{I}_h^i(x,v)=I(x,v)+h^2\widehat{I}_2^i(x,v)+h^4\widehat{I}_4^i(x,v)+\cdots,$$
		satisfy
		\begin{equation*}
			\frac{d}{dt}\widehat{I}_h^e{(z,\dot{z})}=\mathcal{O}(h^N)
		\end{equation*}
		along solutions of the modified differential equation \eqref{modi-EXS} for EXS-O2,
		and
		\begin{equation*}
			\frac{d}{dt}\widehat{I}_h^i{(z,\dot{z})}=\mathcal{O}(h^N)
		\end{equation*}
		along solutions of the modified differential equation \eqref{modi-IMS} for IMS-O2. {Here the functions $\widehat{I}_{2j}^e(x,v)$ and $\widehat{I}_{2j}^i(x,v)$ are $h$-independent and $\mathcal{O}(h^N)$ refers to the truncation term.}
	\end{cor}
	\begin{proof}
		{On the basis of }the conditions above, $-(\ddot{z}\times b)^\intercal E(z)=(\ddot{z}\times b)^\intercal (Qz+q)=\frac{d}{dt}\left((\dot{z}\times b)^\intercal (Qz+q)\right).$
		The hand-side of \eqref{B-mag-mom-IMS} can be written as
		\begin{align*}
			&\frac{h}{2\abs{B} \tan\left(\frac{h}{2}\abs{B}\right)}(\ddot{z}\times b)^\intercal\left[(Qz+q)+\frac{h^2}{4}Q\left(\ddot{z}+\frac{h^2}{12}\ddddot{z}+\cdots\right)\right]\\
			&-\frac{1}{\abs{B}}(\ddot{z}\times b)^\intercal \left[\frac{h^2}{4}Q\left(\dot{z}+\frac{h^2}{6}\dddot{z}+\cdots\right)\right]\times b+\mathcal{O}(h^N).
		\end{align*}
{Since $\widehat{B}Q$ is skew-symmetric and $\widehat{B}^2Q$ is symmetric, we have}
		\begin{align*}
			&(\ddot{z}\times b)^\intercal Qz^{(2k)}\\
			&=-{\ddot{z}}^\intercal \widehat{B}Qz^{(2k)}=\begin{cases}
				0, & \text{$k=1$}\\
				-\frac{d}{dt}\left({\ddot{z}}^\intercal \widehat{B} Qz^{(2k-1)}-{\dddot{z}}^\intercal \widehat{B} Qz^{(2k-1)}+\cdots+(-1)^k\left({z^{(k)}}\right)^\intercal \widehat{B} Qz^{(k+1)}\right), & \text{$k\in \mathbb{N}^*-\{1\}$}
			\end{cases}
		\end{align*}
		and
		\begin{align*}
			&(\ddot{z}\times b)^\intercal \left(Qz^{(2k+1)}\right)\times b\\
			&=-\ddot{z}^\intercal\widehat{B}^2Qz^{(2k+1)}=
			\begin{cases}
				-\frac{1}{2}\dot{z}^\intercal\widehat{B}^2Q\dot{z}, & \text{$k=0$}\\
				-\frac{d}{dt}\left(\ddot{z}^\intercal\widehat{B}^2Qz^{(2k)}-\dddot{z}^\intercal\widehat{B}^2Qz^{(2k-1)}+\cdots+\frac{(-1)^{(k+1)}}{2}\left({z^{(k+1)}}\right)^\intercal\widehat{B}^2Qz^{(k+1)}\right), & \text{$k\in\mathbb{N}^*.$}
			\end{cases}
		\end{align*}
	   {For the sake of formal unity, {we formulate $-{z^{(l)}}^\intercal\widehat{B}^2Qz^{(m)}$}  as $\left(z^{(l)}\times b\right)^\intercal \left(Qz^{(m)}\right)\times b$ for any integers $l$ and $m$.} Hence,
		\begin{equation*}
			\begin{aligned}
				&\frac{d}{dt}\left[I(z,\dot{z})+\frac{h^2}{12\abs{B}}(\ddot{z}\times b)^\intercal (\ddot{z}\times b)-\frac{h}{2\abs{B}\tan\left(\frac{h}{2}\abs{B}\right)}\left(\left(\frac{h^2}{12}(\ddot{z}\times b)^\intercal\dddot{z}+\cdots\right)+(\dot{z}\times b)^\intercal(Qz+q)\right)\right]=\mathcal{O}(h^N)
			\end{aligned}
		\end{equation*}
		along solutions of the modified differential equation \eqref{modi-EXS} for EXS-O2,
		and
		\begin{equation*}
			\begin{aligned}
				&\frac{d}{dt}\left[I(z,\dot{z})+\frac{h^2}{12\abs{B}}(\ddot{z}\times b)^\intercal (\ddot{z}\times b)-\frac{h}{2\abs{B}\tan\left(\frac{h}{2}\abs{B}\right)}\left(\frac{h^2}{12}(\ddot{z}\times b)^\intercal\dddot{z}+\cdots\right)\right]\\
				&-\frac{d}{dt}\left[\frac{h}{2\abs{B}\tan\left(\frac{h}{2}\abs{B}\right)}\left((\ddot{z}\times b)^\intercal(Qz+q)+\frac{h^2}{4}\left(\frac{h^2}{12}(\ddot{z}\times b)^\intercal Q\dddot{z}+\cdots\right)\right)\right]\\
				&+\frac{d}{dt}\left[\frac{h^2}{4\abs{B}}\left(\frac{1}{2}{(\dot{z}\times b)}^\intercal(Q\dot{z})\times b+\cdots\right)\right]=\mathcal{O}(h^N).
			\end{aligned}
		\end{equation*}
	\hfill $ \blacksquare$\end{proof}

\vskip2mm
\textbf{Proof of \eqref{eq-mag-mom}. } {The proof of \eqref{eq-mag-mom} is the same as that of \eqref{eq-energy} and \eqref{eq-momentum},  and we omit it for brevity. }
 \hfill $ \blacksquare$
	\section{Conclusion}\label{conc}
{In this paper, we presented two {splitting} algorithms and studied their long-term behaviour  for solving  charged-particle dynamics.
Using backward error analysis, it was shown that  these two algorithms have  good conservations of energy, momentum and magnetic moment in some special cases. Furthermore, one algorithm was proved to conserve a modified energy {exactly}. All the results were illustrated by some numerical tests. }

\end{document}